\documentclass[a4paper, 11pt]{article} 

\usepackage{amsmath, amscd, amsthm, amssymb, mathrsfs}

\DeclareMathOperator{\Aut}{Aut}

\DeclareMathOperator{\tr}{tr}

\DeclareMathOperator{\U}{U}

\DeclareMathOperator{\GL}{GL}

\newcommand{\R}{\mathbb R}
\newcommand{\C}{\mathbb C}

\newcommand{\N}{\mathbb N}

\newcommand{\diff}{\text{\rm d}}
\newcommand{\del}{\partial}
\newcommand{\delb}{\bar{\del}}

\newcommand{\g}{\mathfrak{g}}
\newcommand{\un}{\mathfrak{u}}

\newcommand{\omegaFS}{\omega_{\mathrm{FS}}}

\renewcommand{\P}{\mathbb P}

\theoremstyle{plain}
	\newtheorem{theorem}{Theorem}
	\newtheorem{proposition}[theorem]{Proposition}
	\newtheorem{lemma}[theorem]{Lemma}

\theoremstyle{definition}
	
	\newtheorem{remark}[theorem]{Remark}

\theoremstyle{plain}
	\newtheorem*{theorem*}{Theorem}
	\newtheorem*{proposition*}{Proposition}
	\newtheorem*{lemma*}{Lemma}
	\newtheorem*{corollary*}{Corollary}
	\newtheorem*{conjecture*}{Conjecture}
	   \newtheorem*{theoremnoname*}{}
\theoremstyle{definition}
	\newtheorem*{definition*}{Definition}
	\newtheorem*{remark*}{Remark}
	\newtheorem*{remarks*}{Remarks}
	
\begin{document}

\title{Calabi flow and projective embeddings}
\author{Joel Fine\footnote{Supported by an FNRS postdoctoral fellowship.}\\
	~\\
	Appendix written by\\
	Kefeng Liu and Xiaonan Ma}
	
\date{ }

\maketitle

\begin{abstract}
Let $X \subset \C\P^N$ be a smooth subvariety. We study a flow, called balancing flow, on the space of projectively equivalent embeddings of $X$ which attempts to deform the given embedding into a balanced one. If $L\to X$ is an ample line bundle, considering embeddings via $H^0(L^k)$ gives a sequence of balancing flows. We prove that, provided these flows are started at appropriate points, they converge to Calabi flow for as long as it exists. This result is the parabolic analogue of Donaldson's theorem relating balanced embeddings to metrics with constant scalar curvature  \cite{donaldson-scape1}. In our proof we combine Donaldson's techniques with an asymptotic result of Liu--Ma \cite{liu-ma-arosnricdg} which, as we explain, describes the asymptotic behaviour of the derivative of the map $\mathrm{FS} \circ \mathrm{Hilb}$ whose fixed points are balanced metrics.
\end{abstract}

\section{Introduction}

\subsection{Overview of results}

The idea of approximating K\"ahler metrics by projective embeddings goes back several years. The fundamental fact---proved by Tian---is that the projective metrics are dense in the space of all K\"ahler metrics. More precisely, let $L\to X^n$ be an ample line bundle over a complex manifold and let $h$ be a Hermitian metric in $L$ whose curvature defines a K\"ahler metric $\omega \in c_1(L)$. Together, $h$ and $\omega$ determine an $L^2$-inner-product on the vector spaces $H^0(L^k)$. Using an $L^2$-orthonormal basis of sections for each $H^0(L^k)$ gives a sequence of embeddings $\iota_k \colon X \to \C\P^{N_k}$ into larger and larger projective spaces and hence a sequence of projective metrics $\omega_k = \frac{1}{k}\iota_k^* \omegaFS$ in the same cohomology class as $\omega$. 

\begin{theorem}[Tian \cite{tian-oasopkmoam}]\label{Tian}
The metrics $\omega_k$ converge to $\omega$ in $C^\infty$ as $k \to \infty$.
\end{theorem} 

From here it is natural to ask if any of the objects studied in K\"ahler geometry can be approximated by objects in projective geometry. An example of this phenomenon, due to Donaldson, is the strong  relationship between balanced embeddings and K\"ahler metrics of constant scalar curvature, which is the central focus of this article. 

Before stating Donaldson's result, we first recall the definition of a balanced embedding, originally due to Luo \cite{luo-gcfgmsopm} and Zhang \cite{zhang-harossv} (see also Bour\-guignon--Li--Yau \cite{bourguignon-li-yau-ubftfeoas}). Let $\mu \colon \C\P^N \to i\un(N+1)$ be the Hermitian-matrix valued function, given in homogeneous unitary coordinates by $\mu = (\mu_{\alpha \beta})$ where
$$
\mu_{\alpha \beta}[x_0 : \cdots : x_N]
=
\frac{x_\alpha \bar x_\beta}{\sum |x_\gamma|^2}.
$$
Given a smooth subvariety $X \subset \C\P^N$, we consider the integral of $\mu$ over $X^n$ with respect to the Fubini--Study metric:
$$
\bar \mu = \int_X \mu\, \frac{\omegaFS^n}{n!}.
$$
The subvariety is called \emph{balanced} if $\bar \mu$ is a multiple of the identity. 

Donaldson proved the following:

\begin{theorem}[Donaldson \cite{donaldson-scape1}]\label{Donaldson 1}
Suppose that for all large $k$ there is a basis of $H^0(L^k)$ which gives a balanced embedding $\iota_k\colon X \to \C\P^{N_k}$ and, moreover, that the metrics $\omega_k = \frac{1}{k}\iota_k^*\omegaFS$ converge in $C^\infty$ to a metric $\omega$. Then $\omega$ has constant scalar curvature.
\end{theorem}

\begin{theorem}[Donaldson \cite{donaldson-scape1}]\label{Donaldson 2}
Suppose that $\Aut(X,L)$ is discrete and that the class $c_1(L)$ contains a metric $\omega$ of constant scalar curvature. Then for all large $k$ there is a basis of $H^0(L^k)$ giving a balanced embedding $\iota_k \colon X\to \C\P^{N_k}$ and, moreover, the  metrics $\omega_k=\frac{1}{k}\iota^*_k\omegaFS$ converge in $C^\infty$ to $\omega$.
\end{theorem}

The goal of this article is to prove the parabolic analogue of Donaldson's Theorems. In \cite{calabi-ekm} Calabi introduced a parabolic flow, \emph{Calabi flow}, which one might hope deforms a given K\"ahler metric towards a constant scalar curvature one. The flow is
$$
\frac{\del \omega}{\del t}
=
i\delb\del \,S\left(\omega(t)\right),
$$
where $S$ denotes scalar curvature.

As we will explain, the projective analogue of this is \emph{balancing flow}. Given an embedding $\iota \colon X \to \C\P^N$, let $\bar \mu_0$ denote the trace-free part of $\bar \mu$, so that $\bar \mu_0 = 0$ if and only if the embedding is balanced. The Hermitian matrix $\bar\mu_0$ defines a vector field on $\C\P^N$ and consequently an infinitesimal deformation of the embedding $\iota$. This defines balancing flow:
$$
\frac{\diff \iota}{\diff t} = -\bar\mu_0(\iota).
$$ 

Our results concern the asymptotics of a certain sequence of balancing flows. Let  $h$ be a Hermitian metric in $L$ whose curvature gives a K\"ahler form $\omega \in c_1(L)$. As in the description of Tian's Theorem \ref{Tian}, let $\iota_k$ be the embedding defined via a basis of $H^0(L^k)$ which is orthonormal with respect to the L$^2(h, \omega)$ inner-product and let $\omega_k = \frac{1}{k}\iota_k^*\omegaFS$ be Tian's sequence of projective approximations to $\omega$. For each $k$, we run a sped-up version of the balancing flow, so that $\iota_k(t)$ solves
\begin{equation}\label{balancing flow}
\frac{\diff \iota_k}{\diff t}
=
-2\pi k^2 \bar\mu_0(\iota_k),
\quad
\iota_k(0) = \iota_k.
\end{equation}
We study the sequence $\omega_k(t) = \frac{1}{k}\iota_k(t)^*\omegaFS$ of metric flows, proving the parabolic analogue of Donaldson's Theorems \ref{Donaldson 1} and \ref{Donaldson 2}:

\begin{theorem}\label{main result 1}
Suppose that for each $t\in [0,T]$ the metric $\omega_k(t)$ converges in $C^\infty$ to a metric $\omega(t)$ and, moreover, that this convergence is $C^1$ in $t$. Then the limit $\omega(t)$ is a solution to Calabi flow starting at $\omega$.
\end{theorem}

\begin{theorem}\label{main result 2}
Suppose that the Calabi flow $\omega(t)$ starting at $\omega$ exists for $t\in[0,T]$. Then for each $t$, the metric $\omega_k(t)$ converges in $C^\infty$ to $\omega(t)$. Moreover, this convergence is $C^1$ in $t$. 
\end{theorem} 

\subsection{A moment map interpretation}\label{moment map picture}

The following picture will not be used directly in our proofs, but it gave the original motivation for this work, so it is perhaps worth mentioning briefly here. First, we recall the standard moment map set-up, which consists of a group $K$ acting by K\"ahler isometries on a K\"ahler manifold $Z$, along with an equivariant moment map $m\colon Z \to \g^*$. The action extends to the complexified group $G = K^\C$ giving a holomorphic, but no longer isometric, action. The problem one is interested in is finding a zero of $m$ in a given $G$-orbit. By $K$-equivariance, this becomes a question about the behaviour of a certain function, called the  \emph{Kempf--Ness function}, $F \colon G/K \to \R$ on the symmetric space $G/K$. This function is geodesically convex and its derivative is essentially $m$. Hence there is a zero of the moment map in the orbit if and only if $F$ attains its minimum. A natural way to search for such a minimum is to consider the downward gradient flow of $F$.

In \cite{donaldson-scape1}, Donaldson explains how this is relevant in our situation. On the one hand, scalar curvature can be interpreted as a moment map, an observation due to Donaldson \cite{donaldson-rogtcga4mt} and Fujiki \cite{fujiki-msopamakm}. In this case the symmetric space is the space $\mathcal H$ of positively curved Hermitian metrics in $L$ or, equivalently once a reference metric in $c_1(L)$ is chosen, the space of K\"ahler potentials (see the work of Donaldson \cite{donaldson-sskgahd}, Mabuchi \cite{mabuchi-ssgockm1} and Semmes \cite{semmes-cmaasm} for a description of this symmetric space structure). Finding a zero of the moment map corresponds to finding a constant scalar curvature metric in the given K\"ahler class. In this context, the Kempf--Ness function is Mabuchi's K-energy and the gradient flow is Calabi flow.

On the other hand, in \cite{donaldson-scape1} Donaldson showed how to fit balanced metrics into a finite dimensional moment-map picture. (This picture has been subsequently studied in \cite{phong-sturm-scmmatdp,wang-mmfiasopm}.)  If $L \to X$ is very ample, then every basis of $H^0(L)$ defines an embedding $X \subset \C\P^N$. We consider the space $Z \cong  \GL(N+1)$ of all bases. There is a K\"ahler structure on $Z$, whose definition involves the fact that each point gives an embedding of $X$, and $\U(N+1)$ acts isometrically with a moment map which is essentially $-i\bar\mu_0$. So finding a zero of the moment map corresponds to finding a balanced embedding. This time, the Kempf--Ness function is a normalised version of what is called the ``$F^0$-functional'' by some authors (it is the function denoted $\tilde Z$ in \cite{donaldson-scape2}). The gradient flow on the Bergman space $\mathcal B = \GL/\U$ is balancing flow. 

Taking successively higher powers $L^k$ of $L$ gives a sequence of moment map problems on successively larger Bergman spaces $\mathcal B_k$, each of which lives inside $\mathcal H$. Put loosely, Donaldson's Theorems \ref{Donaldson 1} and \ref{Donaldson 2} say that the zeros of the finite dimensional moment maps in $\mathcal B_k$ converge in $\mathcal H$ to a zero of the infinite dimensional moment map. Theorems \ref{main result 1} and \ref{main result 2} say that provided we choose the finite dimensional gradient flows to start at a appropriate points then they converge to the infinite dimensional gradient flow. 

In fact, the only aspect of this picture that we use directly in the proofs is that balancing flow is distance decreasing on $\mathcal B$, which follows from the fact that it is the downward gradient flow of a geodesically convex function. This was discovered prior to the moment map interpretation by Paul \cite{paul-gaocms} and Zhang \cite{zhang-harossv}. We remark in passing that Calabi--Chen \cite{calabi-chen-tsokm2} have proved that Calabi flow on $\mathcal H$ is distance decreasing using the symmetric space metric of Donaldson--Mabuchi--Semmes. This is strongly suggested by the standard moment map picture, but doesn't follow directly because $\mathcal H$ is infinite-dimensional. 

\subsection{Additional Context}

Calabi suggested in \cite{calabi-ekm} that, when one exists, a constant scalar curvature K\"ahler metric should be considered a ``canonical'' representative of a K\"ahler class. Since this suggestion, such metrics have been the focus of much work. Additional motivation is provided by the conjectural equivalence between the existence of a K\"ahler metric of constant scalar curvature representing $c_1(L)$ and the stability, in a certain sense, of the underlying polarisation $L \to X$. This began with a suggestion of Yau \cite{yau-opig} which was refined by Tian \cite{tian-occfcswpfcc,tian-kemwpsc} and Donaldson \cite{donaldson-scasotv}. 

Calabi flow, meanwhile, has received less attention. This is no doubt due to the fact that, as it is a fourth-order fully-nonlinear parabolic PDE, there are few standard analytic techniques which apply directly. A start is made in the foundational article by Chen--He \cite{chen-he-otcf} which includes a proof of short-time existence and also shows that when a constant scalar curvature metric $\omega$ exists and the Calabi flow starts sufficiently close to $\omega$ then the flow exists for all time and converges to $\omega$. There are also some long time existence results in which a priori existence of a constant scalar curvature metric is replaced by a ``small energy'' assumption. For example, Tosatti--Weinkove \cite{tosatti-weinkove-tcfwsie} show that, assuming $c_1(X)=0$, if the Calabi flow starts at a metric with sufficiently small Calabi energy, the flow exists for all time and converges to a constant scalar curvature metric. Chen--He \cite{chen-he-tcfotfs} have proved a similar result for Fano toric surfaces.

Balanced metrics have been written about several times since their introduction. Independently, Luo \cite{luo-gcfgmsopm} and Zhang \cite{zhang-harossv} have proved that an embedding can be balanced if and only if it is stable in the sense of GIT, giving the projective analogue of the conjectural relationship between constant scalar curvature metrics and stability mentioned above. The restriction on $\Aut(X, L)$ in Donaldson's Theorem \ref{Donaldson 2} has been relaxed by Mabuchi \cite{mabuchi-soekm,mabuchi-aetatthkcfm1}, whilst the picture in \cite{donaldson-scape1} has been related to the Deligne pairing by Phong--Sturm in \cite{phong-sturm-sefakem,phong-sturm-scmmatdp}, an approach which also leads to a sharpening of some estimates. 

\subsection{Overview of proofs}

\subsubsection{Bergman asymptotics}

The key technical result which underpins Theorems \ref{Tian}, \ref{Donaldson 1} and \ref{Donaldson 2} concerns the \emph{Bergman function}. Given a K\"ahler metric $\omega \in c_1(L)$, let $h$ be a Hermitian metric in $L$ with curvature $2\pi i \omega$. Let $s_\alpha$ be a basis of $H^0(L^k)$ which is orthornomal with respect to the $L^2$-inner-product determined by $h$ and $\omega$. The Bergman function $\rho_k(\omega) \colon X \to \R$ is defined by
$$
\rho_k(\omega) = \sum_\alpha |s_\alpha|^2.
$$ 
where $|\cdot |$ denotes the pointwise norm on sections using $h$. 

The central result concerns the asymptotics of $\rho_k$ and is due  to the work of Catlin \cite{catlin-tbkaatot}, Lu \cite{lu-otlototaeptyz}, Tian \cite{tian-oasopkmoam} and Zelditch \cite{zelditch-skaatot}. We state it as it appears  in \cite{donaldson-scape1} (see  Proposition 6 there and the discussion afterwards). 

\begin{theorem}\label{asymptotics of Bergman}
\mbox{ }
\begin{enumerate}
\item
For fixed $\omega$ there is an asymptotic expansion as $k \to \infty$,
$$
\rho_k(\omega)
=
A_0(\omega)k^n + A_1(\omega) k^{n-1} + \cdots,
$$
where $n = \dim X$ and where the $A_i(\omega)$ are smooth functions on $X$ which are polynomials in the curvature of $\omega$ and its covariant derivatives.  
\item
In particular, 
$$
A_0(\omega) = 1, \quad A_1(\omega) = \frac{1}{2\pi} S(\omega).
$$
\item
The expansion holds in $C^\infty$ in that for any $r, M >0$, 
$$
\left\|
\rho_k(\omega) - \sum_{i=0}^M A_i(\omega)k^{n-i}
\right\|_{C^r(X)}
\leq
K_{r, M, \omega} k^{n-M -1}
$$
for some constants $K_{r,M,\omega}$. Moreover the expansion is uniform in that for any $r$ and $M$ there is an integer $s$ such that if $\omega$ runs over a set of metrics which are bounded in $C^s$, and with $\omega$ bounded below, the constants $K_{r,M,\omega}$ are bounded by some $K_{r,M}$ independent of $\omega$.
\end{enumerate}
\end{theorem}

This expansion essentially proves Theorem \ref{Tian}, since $\rho_k$ relates the original metric $\omega$ to the projective approximation $\omega_k$:
$$
\omega = \omega_k + \frac{i}{2k} \delb\delb \log \rho_k.
$$
Tian's Theorem follows from the fact that the leading term in the expansion of $\rho_k$ is constant.

To see the link with balanced embeddings, note that the embedding $\iota_k$ is balanced if and only if $\rho_k$ is constant.  The fact that the second term in the expansion of $\rho_k$ is the scalar curvature of $\omega$ is at the root of the relationship between the asymptotics of balanced embeddings and constant scalar curvature metrics described by Donaldson's Theorems \ref{Donaldson 1} and \ref{Donaldson 2}.

The asymptotics of the Bergman kernel will also be critical in the proofs of Theorems \ref{main result 1} and \ref{main result 2}, but of equal  importance is another asymptotic result, due to Liu--Ma \cite{liu-ma-arosnricdg}. In fact, for our application a slight strengthening of this result is desirable. Profs.\ Liu and Ma were kind enough to provide a proof of this improvement and this appears in the appendix to this article. Liu--Ma's theorem concerns a sequence of integral operators $Q_k$, introduced by Donaldson in \cite{donaldson-snricdg} and defined as follows. Let $B_k(p,q)$ denote the Bergman kernel of $L^k$; in other words, if $s^*$ denotes the section of $(\bar L^k)^*$ which is metric-dual to a section $s$ of $L^k$, then $B_k$ is the section of $L^k \otimes (\bar L^k)^* \to X \times X$ given by
$$
B_k(p,q) = \sum_{\alpha} s_\alpha (p) \otimes s_\alpha^*(q).
$$
(where $s_\alpha$ is an $L^2$-orthonormal basis of holomorphic sections as before). Now define a sequence of functions $K_k\colon X \times X \to \R$ by 
$$
K_k(p,q)
=
\frac{1}{k^n}|B_k(p,q)|^2
=
\frac{1}{k^n}
\sum_{\alpha, \beta}
( s_\alpha, s_\beta)(p) (s_\beta, s_\alpha)(q).
$$
These functions are the kernels for a sequence of integral operators acting on $C^\infty(X)$ defined by
$$
(Q_kf)(p) = \int_X K_k(p,q)f(q) \frac{\omega^n(q)}{n!}.
$$

Liu--Ma's Theorem (following a suggestion of Donaldson \cite{donaldson-snricdg}) relates the asymptotics of the operators $Q_k$ to the heat kernel $\exp(-s\Delta)$ of $(X, \omega)$. 

\begin{theorem}[Liu--Ma \cite{liu-ma-arosnricdg}, see also the appendix]\label{Liu-Ma}
For any choice of positive integer $r$, there exists a constant $C$ such that for all sufficiently large integers $k$ and any $f \in C^\infty(X)$, 
$$
\left\| 
\left(\frac{\Delta}{k}\right)^r\left\{
Q_k(f) 
- 
\exp\left(-\frac{\Delta}{4\pi k}\right)
f
\right\}
\right\|_{L^2}
\leq \frac{C}{k}\left\| f\right\|_{L^2},
$$
where the norms are taken with respect to $\omega$. Moreover, the estimate is uniform in the sense that there is an integer $s$ such that the constant $C$ can be chosen independently of $\omega$ provided $\omega$ varies over a set of metrics which is bounded in $C^s$ and with $\omega$ bounded below.
\end{theorem}

Liu-Ma's original article \cite{liu-ma-arosnricdg} deals with the cases $r=0$ and $r=1$; the remaining cases are considered in the appendix to this article, which also proves the following $C^m$ estimate:

\begin{theorem}[See appendix]\label{converge in Cinfty}
For any choice of positive integer $r$, there exists a constant $C$ such that for all sufficiently large $k$ and for any $f \in C^\infty(X)$, 
$$
\|Q_k(f) - f\|_{C^m} \leq \frac{C}{k} \|f\|_{C^m},
$$
where the norms are taken with respect to $\omega$. Moreover, the estimate is uniform in the sense that there is an integer $s$ such that the constant $C$ can be chosen independently of $\omega$ provided $\omega$ varies over a  set of metrics which is bounded in $C^s$ and with $\omega$ bounded below.
\end{theorem}

Just as the Bergman function $\rho_k$ appears when comparing a K\"ahler metric $\omega$ to its algebraic approximations $\omega_k$, the operators $Q_k$ appear when one relates infinitesimal deformations of the metric $\omega$ to the corresponding deformations of the approximations $\omega_k$. Since the Calabi flow deforms $\omega$, it is clear that this will be of interest to us. 

To see how the operators $Q_k$ arise, let $h(t) = e^{\phi(t)}h$ denote a path of positively curved Hermitian metrics in $L$, giving a path of K\"ahler forms $\omega(t) \in c_1(L)$. The infinitesimal change in the $L^2$-inner-product on $H^0(L^k)$ corresponds to the Hermitian matrix $A$ whose elements are 
$$
A_{\alpha \beta}
=
\int_X 
\left( k \dot \phi + \Delta \dot\phi \right)
(s_\alpha, s_\beta)\,
\frac{\omega^n}{n!}.
$$
The term $k\dot \phi$ is due to the change in the fibrewise metric, whilst $\Delta \dot\phi$ is due to the change in volume form.
The infinitesimal change in $\omega_k$ corresponding to $A$ is given by the potential 
\begin{eqnarray*}
\frac{1}{k}\tr (A\mu)
	&=&
		\int_X \left(\dot \phi + k^{-1}\Delta \dot \phi \right)(p)
			\frac{
					(s_\alpha, s_\beta)(p)(s_\beta, s_\alpha)(q)}
					{\rho_k(q)}
					\,\frac{\omega^n(p)}{n!},\\
	&=& 
		\int_X \left(\dot \phi + k^{-1}\Delta \dot \phi \right)(p)
			\frac{k^n}{\rho_k(q)}K_k(p,q)\,
				\frac{\omega^n(p)}{n!},\\
	&=&
		\left(
		Q_k\big(\dot\phi\big) 
		+
		k^{-1}Q_k(\Delta \dot \phi) 
		\right)
		\left(
		1+O(k^{-1})
		\right). 
\end{eqnarray*}
(The fact that the potential is $\tr (A\mu)$ is essentially a restatement of the fact that $-i\mu$ is a moment map for the action of  $\U(N+1)$ on $\C\P^N$, whilst the factor of $k^{-1}$ is due to the rescaling needed to remain in a fixed K\"ahler class.) It follows from Liu--Ma's results that $\frac{1}{k}\tr(A\mu) \to \dot \phi$ in $C^\infty$ and hence that the convergence of algebraic approximations $\omega_k(t)$ to $\omega(t)$ is also $C^1$ in the $t$ direction.

We can describe this calculation in the notation of \cite{donaldson-scape2}. Recall that $\mathcal H$ denotes the space of positively curved Hermitian metrics in $L$, whilst $\mathcal B_k$ denotes the space of projective Hermitian metrics in $L^k$, i.e., those obtained by pulling back the Fubini--Study metric from $\mathcal O(1) \to \C\P^{N_k}$ using embeddings via $H^0(L^k)$. Given  $h \in \mathcal H$ in $L$, using an $L^2(h)$-orthonormal basis of $H^0(L^k)$ to embed $X$ gives a projection $\mathrm{Hilb}_k\colon \mathcal H \to \mathcal B_k$; meanwhile, taking the $k^{\mathrm{th}}$ root gives an inclusion $\mathrm{FS}_k \colon\mathcal B_k \to \mathcal H$. Composing gives a map $\Phi_k  = \mathrm{FS}_k \circ \mathrm{Hilb}_k \colon \mathcal H \to \mathcal H$ and this calculation shows that the derivative of $\Phi_k$ at a given point $h$ satisfies $(\diff \Phi_k)_h = Q_k + O(k^{-1})$. So, whilst the expansion of $\rho_k$ tells us that $\Phi_k(h)/k^n \to h$, Liu--Ma's asymptotics give that $(\diff \Phi_k)_h(\dot\phi) \to \dot \phi$.

\subsubsection{Outline of arguments}

Our over-all approach follows the general scheme of Donaldson's proofs of Theorems \ref{Donaldson 1} and \ref{Donaldson 2}. We begin in \S\ref{BF converges} by proving Theorem \ref{main result 1}. Just as Donaldson's Theorem \ref{Donaldson 1} is implied more-or-less directly by the uniformity of the asymptotic expansion of the Bergman kernel, so our result will follow easily from this combined with the uniformity in Liu-Ma's Theorem. 

We then move on to the proof of Theorem \ref{main result 2}. As with Theorem \ref{Donaldson 2}, this part requires substantially more effort. It is shown in \S\ref{approximating BF} that the standard sequence of projective approximations to Calabi flow gives an $O(k^{-1})$ approximation to balancing flow. This is analogous to the fact that the standard sequence of projective approximations to a constant scalar curvature metric are themselves close to being balanced. Just as in the case of balanced metrics, however, $O(k^{-1})$ is not strong enough; for later arguments it becomes important to improve this to beat any power of $k^{-1}$. 

Donaldson solved this problem by considering instead a perturbation of the constant scalar curvature metric $\omega$: 
$$
\omega + i \delb\del \sum_{j=1}^m k^{-j}\eta_j
$$
where the potentials $\eta_j$ solve partial differential equations of the form $L\eta = f$ for a certain linear elliptic operator $L$. We apply the same idea to Calabi flow, using time-dependent potentials $\eta_j(t)$ which are required to solve parabolic equations $\dot \eta + L\eta = g$ associated to the same operator $L$.

After this perturbation we have, for each $k$, a path of projective metrics $\omega'_k(t)$ which on the one hand converge to Calabi flow and on the other hand are $O(k^{-m})$ from the balancing flow $\omega_k(t)$. Here, we are working for each $k$ in the Bergman space $\mathcal B_k = \GL(N_k+1)/\U(N_k+1)$ and $O(k^{-m})$ means with respect to the Riemannian distance function $d_k$ of the symmetric space metric. The remainder of the proof is concerned with uniformly controlling the $C^r$ norm on metric tensors by the Riemannian distance $d_k$. This involves two parts: firstly, some analytic estimates proved in \S\ref{analytic estimates} reduce the problem to controlling $\bar\mu$; secondly, some estimates in projective geometry proved in \S\ref{projective estimates} show how to control $\bar\mu$ by $d_k$. \S\ref{completing the proof} puts all the pieces together and completes the proof of Theorem \ref{main result 2}.

\subsection{Acknowledgements}

I am very grateful to G\'abor Sz\'ekelyhidi and Xiuxiong Chen for many helpful discussions during the course of this work. I am also indebted to Kefeng Liu and Xiaonan Ma for providing me with the version of Theorem \ref{Liu-Ma} used here and kindly agreeing to write the necessary additional arguments in an appendix to this article. I would also like to thank Julien Keller, Yanir Rubinstein and Richard Thomas for conversations about Bergman asymptotics and balanced embeddings.

\section{When the balancing flows converge}\label{BF converges}

In this section we will prove Theorem \ref{main result 1}. We begin by describing balancing flow in terms of K\"ahler potentials. Given an embedding $\iota \colon X \to \C\P^N$ and a Hermitian matrix $A \in i\un(N+1)$, we view $A$ as a vector field $\xi_A$ on $\C\P^N$ and consequently as an infinitesimal perturbation of $\iota$. The corresponding infinitesimal change in $\iota^*\omegaFS$ is given by the potential  $\tr(A\mu)$ restricted to $X$ via $\iota$. (It suffices to prove this for $\C\P^N$ itself, where it follows from the fact that $-i\mu \colon \C\P^N \to \un(N+1)$ is a moment map for the $\U(N+1)$-action.) Accordingly, the potential corresponding to balancing flow is the \emph{balancing potential} $\beta = -\tr(\bar\mu_0 \mu)$. 

To obtain the correct asymptotics when considering embeddings via higher and higher powers $L^k$ it is necessary to rescale the balancing flow to be generated by $-2\pi k^2\bar\mu_0$. Since the restriction of the Fubini--Study metric is also rescaled to remain in the fixed class $c_1(L)$, the corresponding balancing potential is $\beta_k = -2\pi k \tr (\bar\mu_0 \mu)$. 

Theorem \ref{main result 1} will follow directly from the next result.

\begin{theorem}\label{limit of balancing potentials}
Let $h_k\in \mathcal B_k$ be a sequence of Bergman metrics whose rescaled curvatures $\omega_k \in c_1(L)$ converge in $C^\infty$ to a metric $\omega$. Then the balancing potentials converge in $C^\infty$ to the potential generating Calabi flow at $\omega$:
$$
\beta_k(\omega_k)\to S(\omega) - \bar S.
$$
(Here $\bar S$ is the mean value of the scalar curavture $S$.)
\end{theorem}

\begin{proof}
Let $s_\alpha$ be an orthonormal basis of $H^0(L^k)$ with respect to the $L^2$-inner-product determined by $h_k$ and $\omega_k$. Then the balancing potential is 
$$
\beta_k(\omega_k)(p)
=
2\pi k\int_X\sum\left(
\frac{\delta_{\alpha\beta}}{N_k+1}
-
\frac{(s_\alpha, s_\beta)(q)}{\rho_k(\omega_k)(q)}
\right)
\frac{(s_\beta, s_\alpha)(p)}{\rho_k(\omega_k)(p)}
\,\frac{\omegaFS^n(q)}{n!}.
$$
Here $N_k+1 = \dim H^0(L^k)$ which, by Riemann--Roch and Chern--Weil, is a polynomial with leading terms
$$
N_k+1 = V\left(k^n + \frac{\bar S}{2\pi}k^{n-1} + \cdots \right),
$$
where $V = c_1(L)^n$.

Since $\omega_k$ converges in $C^\infty$, we can apply the uniform asymptotic expansion of the Bergman functions $\rho_k(\omega_k)$ (Theorem \ref{asymptotics of Bergman}) to conclude that $\frac{1}{k}\omegaFS= \omega_k + O(k^{2})$ and, moreover, that
$$
\rho_k(\omega_k) = k^n + \frac{S(\omega_k)}{2\pi} k^{n-1} + O(k^{n-2}).
$$
From here on in the proof, we write $\rho_k = \rho_k(\omega_k)$; similarly  we write the operators appearing in Liu--Ma's Theorem as $Q_k = Q_k(\omega_k)$. 

Using the uniform expansion of the Bergman function, we have
\begin{eqnarray*}
\beta_k(\omega_k)
&=&
\frac{2\pi Vk^{n+1}}{N_k+1}
-
\frac{2\pi k^{n+1}}{\rho_k(p)}
\int_X
K_k(p,q)\left(\frac{k^n}{\rho_k(q)} + O(k^{-2})\right)\,\frac{\omega_k^n(q)}{n!}
\\
&=&
\left(2\pi k -\bar S + O(k^{-1})\right)\\
&&
\quad\quad
-
\left(2\pi k - S(\omega_k) + O(k^{-1})\right) Q_k(1 + O(k^{-1}))\\
&=&
\left(S(\omega_k) - \bar S \right)\left(1 + O(k^{-1}) + Q_k(O(k^{-1}))\right)
\end{eqnarray*}
since, by definition, $Q_k(1) = \rho_k/k^n = 1 + O(k^{-1})$. We will prove this converges to $S - \bar S$ by showing that the error $Q_k(O(k^{-1}))$ converges to zero in $C^\infty$.

To control the term $Q_k(O(k^{-1}))$ we have to be a little careful, since the convergence $Q_k(f) \to f$ is not uniform. This 
can be seen for the heat kernel $\exp(-\Delta/k)$ itself by considering eigenfunctions of the Laplacian with higher and higher eigenvalues. Instead, for $f_k = O(k^{-1})$ note that by the $r=0$ case of Theorem \ref{Liu-Ma},
\begin{eqnarray*}
\| Q_k(f_k))\|_{L^2}
&\leq &
\frac{C}{k} \|f_k\|_{L^2}
+
\left\| \exp\left(-\frac{\Delta}{4\pi k}\right)f_k\right\|_{L^2}\\
&\leq& \left(\frac{C}{k} + 1\right) \| f_k\|
\end{eqnarray*}
where the last line follows because the heat kernel reduces $L^2$-norm. So, indeed, $Q_k(f_k) \to 0$ in $L^2$. (Note we have used the uniformity in Theorem~\ref{Liu-Ma} to ensure that the constant $C$ appearing here from the estimate for the asymptotics of the operators $Q_k(\omega_k)$ can be taken independent of $k$, since $\omega_k \to \omega$ in $C^\infty$.)

For convergence in $L^2_{2r}$, write 
$$
f_k = k^{-1}F_1 + \cdots k^{-r+1}F_{r-1} + \hat f_k
$$
where the $F_j$ are functions independent of $k$ and $\hat f_k = O(k^{-r})$. Now Theorem~\ref{converge in Cinfty} guarantees that $k^{-j}Q_k(F_j) \to 0$ in $C^{\infty}$, whilst Theorem \ref{Liu-Ma} gives
\begin{eqnarray*}
\left\|
\Delta^r \left(Q_k(\hat f_k)\right)
\right\|_{L^2}
&\leq&
Ck^{r-1}\|\hat f_k\|_{L^2}
+
\left\| \exp\left(-\frac{\Delta}{4\pi k} \right) \left(\Delta^r \hat f_k\right)\right\|_{L^2}\\
&\leq&
Ck^{r-1}\|\hat f_k\|_{L^2}
+
\|\Delta^r \hat f_k\|_{L^2}\\
&\leq&
\left(Ck^{r-1} + 1\right)\left \|\hat f_k\right\|_{L^2_{2r}}
\end{eqnarray*}
Since $\hat f_k = O(k^{-r})$, we see that $Q_k(\hat f_k) \to 0$ in $L^2_{2r}$ and, hence, that $Q_k(f_k) \to 0$ in $C^\infty$. 

Recall that
$$
\beta_k(\omega_k) = \left(S(\omega_k) - \bar S\right) \left(1 + O(k^{-1}) + Q_k(O(k^{-1})) \right).
$$
Now combine the fact that $Q_k(O(k^{-1})) \to 0$ in $C^\infty$ with the fact that $S(\omega_k) \to S(\omega)$ in $C^\infty$ to complete the proof.

\end{proof}

We now give the proof of Theorem \ref{main result 1}. Recall that $\omega \in c_1(L)$  is a given K\"ahler form, $\omega_k$ is the sequence of projective metrics in Tian's Theorem \ref{Tian} and $\omega_k(t)$ is the balancing flow starting at $\omega_k$. We assume that for each value of $t\in [0,T]$, $\omega_k(t)$ converges in $C^\infty$ to a metric which we denote $\omega(t)$ and that moreover the convergence is $C^1$ in $t$. So, if we write $\del \omega/\del t = i\delb\del f(t)$ for a function $f(t)$ with $\omega(t)$-mean-value zero, then $\beta_k(\omega_k(t)) \to f(t)$. Now applying Theorem~\ref{limit of balancing potentials} to $\omega_k(t)$ it follows that $f(t)= S(\omega(t)) - \bar S$ and so $\omega(t)$ is a solution to Calabi flow on $[0,T]$ starting at $\omega$.

\section{Using Calabi flow to approximate balancing flow}\label{approximating BF}

We now move on to the proof of Theorem \ref{main result 2}, a task which will take up the remainder of the article. 

\subsection{First order approximation}

We first explain how Calabi flow can be used to approximate the balancing flow metrics $\omega_k(t)$ to $O(k^{-1})$. Recall that $h$ is a Hermitian metric in $L$ with curvature $2\pi i\omega$. Since we are assuming Calabi flow exists, we have a path $h(t)=e^{\phi(t)}h$ of Hermitian metrics with curvatures $\omega(t)$ such that
$$
\dot \phi = S(\omega(t)) - \bar S.
$$
Let $\hat h_k(t) \in \mathcal B_k$ denote the $k^\mathrm{th}$ Bergman point of $h(t)$. In other words, take a basis $s_\alpha$ of $H^0(L^k)$ which is orthonormal with respect to the $L^2$-inner-product determined by $h(t)$ and $\omega(t)$; this defines a projective metric $\hat h_k(t)$ in $L^k$ which is either characterised as the pull back to $L^k$ of the Fubini--Study metric from $\mathcal O(1) \to \C\P^{N_k}$ via the embedding determined by $s_\alpha$ or, equivalently, as the unique metric for which $\sum |s_\alpha|^2=1$. Meanwhile, we denote by $h_k(t)\in\mathcal B_k$ the balancing flow starting at $\hat h_k(0)$. We will estimate the distance in $\mathcal B_k$ between $h_k(t)$ and $\hat h_k(t)$ using the symmetric Riemannian metric. It turns out to be convenient to rescale this by a power of $k$. Denote by $d_k$ the distance function arising from using the metric given by the rescaled Killing form $k^{-(n+2)}\tr A^2$. (It can been shown that these norms converge in a certain sense to the $L^2$-norm on potentials, so this is a natural rescaling to consider.)

Our goal in this subsection is to prove the following result.

\begin{proposition}\label{first approximation}
There is a constant $C$ such that for all $t \in [0,T]$, 
$$
d_k\left(h_k(t), \hat h_k(t) \right) \leq \frac{C}{k}.
$$
\end{proposition}

\begin{proof}
We begin by considering the tangent vector to $\hat h_k(t)$. In general, given a smooth path $h(t)=e^{\phi(t)}h_0$ of positively curved Hermitian metrics, the infinitesimal change in $L^2$-inner-product on $H^0(L^k)$ corresponds to the Hermitian matrix $U = (U_{\alpha \beta})$ where
\begin{equation}\label{tangent to path}
U_{\alpha \beta}
=
\int_X (s_\alpha, s_\beta)\left(k\dot\phi + \Delta \dot \phi \right) \frac{\omega^n}{n!}.
\end{equation}
Here $s_\alpha$ is an $L^2(h^k(t), \omega(t))$-orthonormal basis of $H^0(L^k)$ and all relevant quantities are computed with respect to $h(t)$ and $\omega(t)$. The term $k\dot\phi$ here corresponds to the infinitesimal change to the fibrewise metric whilst the term $\Delta \dot \phi$ corresponds to the infinitesimal change in volume form. 

In our case, this gives that the tangent $U_k(t)$ to $\hat h_k(t)$ is the Hermitian matrix
$$
U_k = \int_X (s_\alpha, s_\beta)\left(k (S- \bar S) + \Delta S \right) \frac{\omega^n}{n!},
$$
where $S$ is the scalar curvature of $\omega(t)$.

Meanwhile, the tangent to balancing flow through the same point $\hat h_k(t)$ is the Hermitian matrix
$$
V_k
=
2 \pi k^2 \int_X \left( 
\frac{\delta_{\alpha \beta}}{N_k+1} - \frac{(s_\alpha, s_\beta)}{\rho_k} 
\right)
\frac{\omegaFS^n}{n!},
$$
where $\rho_k$ is the Bergman function for $h(t)$. 

Using the asymptotic expansion of $\rho_k$ and the fact that $\omega(t) = \frac{1}{k}\omegaFS + O(k^{-2})$, there is an asymptotic expansion of $V_k$:
\begin{equation}\label{tangent to BF}
V_k
=
\int_X (s_\alpha, s_\beta)( k(S - \bar S) + O(1)) \frac{\omega^n}{n!}.
\end{equation}
So
$$
U_k-V_k
=
\int_X (s_\alpha, s_\beta)\, O(1) \frac{\omega^n}{n!}.
$$ 
It follows that we can write the norm of $U_k - V_k$ in terms of the operators $Q_k = Q_k(\omega(t))$ appearing in Liu--Ma's Theorem \ref{Liu-Ma}: 
$$
\frac{\tr\left(U_k - V_k\right)^2}{k^{n+2}}
=
\int_{X\times X}
K_k(p,q) G_k(p)G_k(q)
=
\langle G_k, Q_k(G_k) \rangle_{L^2}
$$
where $G_k = O(k^{-1})$.

Now, denoting the $L^2$-norm by $\| \cdot \|$, 
\begin{eqnarray*}
\langle G_k, Q_k(G_k) \rangle_{L^2}
&\leq&
\|G_k\| \|Q_k(G_k)\|, \\
&\leq&
\|G_k\| \left( \frac{C}{k}\|G_k\| + \|\exp(-\Delta/(4\pi k)) G_k \| \right), \\
&\leq&
\|G_k\|^2 \left(\frac{C}{k} + 1\right),\\
&=&
O(k^{-2}).
\end{eqnarray*}
(The penultimate inequality uses the fact that the heat kernel reduces the $L^2$-norm.) So $U_k - V_k$ is $O(k^{-1})$ in the rescaled symmetric Riemannian metric on $\mathcal B_k$. Moreover, the bound is uniform in $t$ because of the uniformity in the asymptotic behaviour of $\rho_k$ and $Q_k$. This amounts to the infinitesimal version of the result we are aiming for. 

In order to prove the actual result, let
$$
f_k(t) = d_k\left(h_k(t), \hat h_k(t)\right)
$$
denote the distance we are trying to control. Let $\tilde h_k(t)$ denote the balancing flow which at time $t= t_0$ passes through the point $\hat h(t_0)$. Our bound on $U_k-V_k$ says that $\hat h_k$ and $\tilde h_k$ are tangent to $O(k^{-1})$ at $t=t_0$. Now balancing flow is the downward gradient flow of a geodesically convex function, hence is distance decreasing (this follows from the general moment map description alluded to in \S\ref{moment map picture}; it was discovered prior to the moment map interpretation by Paul \cite{paul-gaocms} and Zhang \cite{zhang-harossv}). So $\tilde h_k$ and $h_k$ get closer and closer together. It follows that there is a constant $C$ such that for all $k$, at $t=t_0$,
$$
\frac{\diff f_k}{\diff t} 
\leq \frac{C}{k},
$$
However, $t_0$ was arbitrary, hence $f_k$ has sub-linear growth on $[0,T]$ and, moreover, $f_k(0)=0$ so $f_k(t) \leq CT/k$ for all $t$.
\end{proof}

\subsection{Higher order approximations}

Unfortunately, the fact that the images of Calabi flow in $\mathcal B_k$ approximate balancing flow to $O(k^{-1})$ with respect to $d_k$ is not sufficient for us to show that the balancing flows converge to Calabi flow. This is similar to the problem encountered by Donaldson and we resolve it by the parabolic analogue of the trick appearing in \S 4.1 of \cite{donaldson-scape1}. Namely, we perturb the Calabi flow $h(t) = e^{\phi(t)}h$ by a polynomial in $k^{-1}$ and consider instead a sequence of flows indexed by $k$. Let
$$
\psi(k; t)
=
\phi(t) + \sum_{j=1}^m k^{-j} \eta_j(t)
$$
where $\eta_j(t)$ are some judiciously chosen time-dependent potentials. Denote by $h(k;t) = e^{\psi(k;t)}h$ the corresponding sequence of paths of Hermitian metrics. Their curvatures give the perturbation
$$
\omega(k;t)
=
\omega(t) + i\delb\del \sum_{j=1}^m k^{-j} \eta_j(t)
$$
of Calabi flow on the level of K\"ahler forms. Note that for any given choice of $\eta_j$, $\omega(k;t)$ is positive for large enough $k$.

Let $h'_k(t) \in \mathcal B_k$ denote the $k^\mathrm{th}$ Bergman point of $h(k;t)$; i.e., given a basis $s_\alpha$ of $H^0(L^k)$ which is orthonormal with respect to the $L^2$-inner-product determined by $h(k;t)$ and $\omega(k;t)$, $h'_k(t)$ is pull-back of the Fubini--Study metric in $\mathcal O(1) \to \C\P^{N_k}$ or, equivalently, the unique Hermitian metric in $L^k$ such that $\sum |s_\alpha|^2 =1$.

Our goal in this subsection is to prove:

\begin{theorem}\label{higher order approximation}
For any $m$, there exist functions $\eta_1, \ldots, \eta_m$ and a constant $C$ such that the perturbed Calabi flow $h(k;t) = e^{\psi(k;t)}h$ with
$$
\psi(k; t)
=
\phi(t) + \sum_{j=1}^m k^{-j} \eta_j(t)
$$
satisfies for all $t \in [0,T]$ and all $k$,
$$
d_k\left(h_k(t), h'_k(t)\right) \leq \frac{C}{k^{m+1}}
$$
\end{theorem}

\begin{proof}
The proof is by induction with Proposition \ref{first approximation} providing the case $m=0$. For clarity, we explain first the case $m=1$ in detail before moving to the general inductive step. So, let $\psi(k; t) = \phi(t) + k^{-1} \eta(t)$ for some $\eta$ which we will now explain how to find.

Let $A_k(t)$ denote the Hermitian matrix corresponding to the tangent to the path $h'_k(t)$. From equation (\ref{tangent to path}), we have that
$$
A_k
=
\int_X (s_\alpha, s_\beta)\left(
k (S(\omega(t))-\bar S) + \dot \eta + \Delta S + O(k^{-1})
\right)
\frac{\omega(k;t)^n}{n!}.
$$
Here $S(\omega(t)) - \bar S$ is the tangent of the \emph{unperturbed} Calabi flow $h(t)$ and the Lapalcian is that of the \emph{unperturbed} metric $\omega(t)$;  the $O(k^{-1})$ terms involve $k^{-1}\Delta \dot\eta$ and also the fact that in the full formula, the Laplacian of $\omega(k;t)$ should appear, but this agrees with $\Delta$ to $O(k^{-1})$. 

Next we compute the asympototics of the Hermitian matrix $B_k(t)$ which is tangent to balancing flow through the point $h'_k(t)$. Let $L_t$ denote the linearisation of the scalar curavture map associated to $\omega(t)$; so $S(\omega(k;t)) = S(\omega(t)) + k^{-1}L_t(\eta(t)) + O(k^{-2})$. $B_k$ is given, as in equation (\ref{tangent to BF}), by the following expression, where we have explicitly notated the $O(1)$ term by $F$:
\begin{eqnarray*}
B_k
&=&
\int_X
(s_\alpha, s_\beta)
\Big(
k\left( S(\omega(k;t) - \bar S)\right) + F + O(k^{-1}) 
\Big)
\frac{\omega(k;t)^n}{n!}\\
&=&
\int_X
(s_\alpha, s_\beta)
\Big(
k\left(S(\omega(t) - \bar S)\right)  - L_t(\eta(t)) + F + O(k^{-1})  
\Big)
\frac{\omega(k;t)^n}{n!}
\end{eqnarray*}
Here we use the uniformity in Theorem \ref{asymptotics of Bergman} along with the fact that $\omega(k;t) \to \omega(t)$ in $C^\infty$ when expanding $\rho_k(\omega(k;t))$. It follows that 
$$
A_k - B_k
=
\int_X (s_\alpha, s_\beta)
\Big[
\dot\eta + L_t(\eta) - F + \Delta S
+ O(k^{-1})
\Big]
\frac{\omega(k;t)^n}{n!}.
$$

Now we chose $\eta$ to solve the Cauchy problem for the inhomogeneous, non-autonomous, linear, parabolic evolution equation:
\begin{equation}\label{linear parabolic}
\dot\eta + L_t(\eta) 
=
F - \Delta S
\end{equation}
for $t\in [0,T]$, with initial condition $\eta(0)=0$. It is standard that equation (\ref{linear parabolic}) has a solution provided the spectra of the operators $L_t$ are bounded below. The lower bound on the spectra ensures that for each $t$, $-L_t$ generates an analytic strongly continuous semi-group and from here the existence of a solution to equation (\ref{linear parabolic}) follows from semi-group theory. See, for example, the texts \cite{amann-laqpp} or \cite{zhang-harossv}. To verify that each of the operators $L_t$ have only finitely many negative eigenvalues, we use the fact that
$$
L_t(\eta)
=
\mathcal D^* \mathcal D(\eta) - (\nabla \eta , \nabla S)
$$
which appears, for example, in \cite{donaldson-rogtcga4mt}. For our purposes, all that matters in this expression is that the first term $\mathcal D^*\mathcal D$ is non-negative, elliptic and of higher order than the second term involving gradients. This means we can connect $L_t$ by a path of elliptic operators $L_t(s)$ to the non-negative operator $\mathcal D^* \mathcal D$:
$$
L_t(s)(\eta) 
=
\mathcal D^*\mathcal D(\eta) - s(\nabla \eta, \nabla S).
$$
As $s$ runs from $0$ to $1$ it is standard that only finitely many eigenvalues of $L_t(s)$ can become negative, proving that the spectrum of $L_t = L_t(1)$ is bounded below. 

With this choice of $\eta$, tracing through the argument used in the proof of Proposition \ref{first approximation} we see that
$$
\frac{\tr\left(A_k - B_k\right)^2}{k^{n+2}}
=
\int_{X\times X}
K_k(p,q) G_k(p)G_k(q)
=
\langle G_k, Q_k(G_k) \rangle_{L^2}
$$
where this time $G_k = O(k^{-2})$. As before, it follows from Liu--Ma's theorem that there is a constant $C$ such that, for all $t\in [0,T]$,
$$
\frac{\tr\left(A_k - B_k\right)^2}{k^{n+2}}
\leq
\frac{C}{k^4}.
$$
Throughout, we have expanded $\rho_k(\omega(k;t))$ and $Q_k(\omega(k;t))$ using the uniformity in Theorems \ref{asymptotics of Bergman} and \ref{Liu-Ma} along with the fact that $\omega(k;t) \to \omega(t)$ in $C^\infty$, uniformly for $t\in[0,T]$. This is the infinitesimal version of the result with $m=1$. As in the proof of Proposition \ref{first approximation} this implies that there is a constant $C$ such that 
$$
d_k\left(h_k(t), h'_k(t)\right) \leq \frac{C}{k^2}
$$
and so the result with $m=1$ is true.

For general $m$, we work iteratively and assume we have selected $\eta_j$ for $j=1, \ldots m-1$ solving a collection of linear parabolic evolution equations to be specified. Let 
$$
\psi(k; t) =\phi(t) + \sum_{j=1}^mk^{-j}\eta_j
$$ 
where we will find $\eta_m$ presently. Using equation (\ref{tangent to path}) we have that the tangent to the path $h'_k(t)$ is
$$
A_k
=
\int_X(s_\alpha, s_\beta)
\left( k(S - \bar S) 
+ 
\sum_{j=0}^{m-1} k^{-j}\dot \eta_{j+1} 
+ 
\Delta' S
+
\sum_{j=1}^m k^{-j} \Delta' \dot \eta_j
\right)\, \frac{\omega^n(k;t)}{n!}
$$
where $S =S(\omega(t))$ and $\Delta'$ is the Laplacian of the metric $\omega(k;t)$. 

The Laplacian depends analytically on the metric meaning that we can write $\Delta'$ as a power series in $k^{-1}$:
$$
\Delta' = \Delta_0 + k^{-1} \Delta_1 + \cdots
$$
where $\Delta_0$ is the Laplacian of the unperturbed metric $\omega(t)$ and $\Delta_r$ depends only on $\eta_j$ for $j=1, \ldots r$. This means that the Laplacian terms in the integrand for $A_k$ expands futher as
\begin{eqnarray*}
\Delta' S 
&=&
\sum_{j=0}^{m-1}
k^{-j}\Delta_j S + O(k^{-m})\\
\sum_{j=1}^m k^{-j} \Delta'\dot \eta_j
&=&
\sum_{j+r=1}^{m-1} k^{-j-r} \Delta_r \dot \eta_j
+
O(k^{-m})
\end{eqnarray*}
Crucially, the choice of $\eta_m$ only affects the $O(k^{-m})$ terms in these two expansions and no lower order terms. So, up to $O(k^{m-1})$, the only contribution of $\eta_m$ to $A_k$ is the term involving $k^{-m+1}\dot \eta_m$. Hence we can write 
$$
A_k = \int_X (s_\alpha, s_\beta) \left(k(S - \bar S) + \sum_{j=0}^{m-1}k^{-j}M_j + k^{-m+1}\dot \eta_m + O(k^{-m})\right)\frac{\omega^n(k;t)}{n!}
$$
where $S=S(\omega(t))$ and the $M_j$ are determined by the $\eta_j$ for $j < m$.

Meanwhile, as in equation (\ref{tangent to BF}), 
$$
B_k 
= 
\int_X(s_\alpha, s_\beta) 
\Big(k\left(S(\omega(k;t)) - \bar S)\right) + \Phi_k \Big)\, \frac{\omega(k;t)^n}{n!}
$$
where $\Phi_k $ is built out of the Bergman function $\rho_k(\omega(k;t))$ by a combination of $\rho_k^{-1}$ and errors introduced by replacing $\omegaFS$ with $\omega(k;t)$. Theorem \ref{asymptotics of Bergman} says that $\rho_k$ has an asymptotic expansion in which the coefficients are polynomials in the curvature of $\omega(k;t)$. Consequently, $\Phi_k$ has an asymptotic expansion, this time in increasing powers of $k^{-1}$, and again the coefficients are polynomials in the curvature of $\omega(k;t)$. It follows that the first contribution of $\eta_m$ to $\Phi_k$ occurs at $O(k^{-m})$. In addition, scalar curvature depends analytically on the metric, so again we have that the only contribution of $\eta_m$ to $S(\omega(k;t))$ occurs at $O(k^{-m})$ and here the contribution is precisely $k^{-m}L_t(\eta_m)$. So we can write $B_k$ as
$$
\int_X(s_\alpha, s_\beta)
\left( k( S - \bar S) + \sum_{j=0}^{m-2}k^{-j}F_j + k^{-m+1}(F_m- L_t(\eta_m) ) + O(k^{-m})\right)
$$
where $S = S(\omega(t))$ and all $F_j$ for $j<m$ are determined by $\eta_j$ for $j <m$.

To complete the proof, we assume that we have chosen the $\eta_1, \ldots, \eta_{m-1}$ so that the terms of $O(k^{-m+2})$ in $A_k - B_k$ cancel and, moreover, so that $\eta_j(0)=0$. This amounts to solving a sequence of parabolic Cauchy problems of the form (\ref{linear parabolic}) in which the inhomogeneous term in the $j^\mathrm{th}$ equation involves the solutions to all previous equations. Assuming this is done we are left with 
$$
A_k - B_k 
= 
\int_X (s_\alpha, s_\beta)\Big[k^{-m+1}(\dot \eta_m + L_t(\eta_m) + M_m - F_m) + O(k^{-m})\Big]\frac{\omega^n(k;t)}{n!}
$$
for $M_m$ and $F_m$ depending only on $\eta_j$ for $j<m$. Choosing $\eta_m$ to solve the parabolic equation
$$
\dot \eta_m + L(\eta_m) = F_m - M_m
$$
with $\eta_m(0)=0$ gives 
$$
A_k -B_k = \int_X (s_\alpha, s_\beta) O(k^{-m}) \, \frac{\omega(k;t)^n}{n!}
$$
and from here the proof proceeds via Liu--Ma's theorem precisely as above.
\end{proof}

\section{Analytic estimates}\label{analytic estimates}

The previous section produced, for a given integer $m$, a sequence of flows $\omega(k; t)$ for which $\omega(k;t) \to \omega(t)$ as $k\to\infty$ in $C^\infty$ and such that the $k^\mathrm{th}$ Bergman point $h'_k(t)$ of $\omega(k;t)$ satisfies $d_k(h'_k(t), h_k(t)) = O(k^{-m-1})$. To complete the proof that $\omega_k(t) \to \omega(t)$ in $C^\infty$ we use the fact that, in the regions of $\mathcal B_k$ of interest to us at least, $k^{(r/2) + 1+ n}d_k$ uniformly controls the $C^r$ norm on the curvature tensors of Bergman metrics. It is precisely this power of $k$ appearing in front of $d_k$ which makes the higher order approximations of Theorem \ref{higher order approximation} necessary. 

The first step in controlling the $C^r$ norm, carried out in this section, is to prove some analytitc estimates which reduce the problem to controlling the norm of the matrix $\bar\mu$. The main estimate we use is due to Donaldson \cite{donaldson-scape1}. We give here a brief description of the relevant part of \S 3.2 of \cite{donaldson-scape1}. In order to avoid worrying about powers of $k$ at every step here, when proving the estimates we use for each $k$ the \emph{large} metrics in the class $k c_1(L)$. Then, at the end it is a simple matter to rescale to metrics in the fixed class and take care of the powers of $k$ at a single stroke. 

Fix a reference metric $\omega_0 \in c_1(L)$ and denote $\tilde \omega_0 = k \omega_0 \in kc_1(L)$. We say another metric $\tilde \omega \in kc_1(L)$ has \emph{$R$-bounded geometry in $C^r$} if $\tilde \omega > R^{-1} \tilde \omega_0$ and
$$
\| \tilde \omega - \tilde \omega_0 \|_{C^r} < R
$$
where the norm $\|\cdot \|_{C^r}$ is that determined by the metric $\tilde \omega_0$.
Given a basis $\{s_\alpha\}$ for $H^0(L^k)$ we get an embedding $X \subset \C\P^{N_k}$ and hence a metric $\tilde \omega =\omegaFS|_X$. Equivalently, $2\pi i\tilde \omega$ is the curvature of the unique metric on $L^k$ for which $\sum |s_\gamma|^2 = 1$. We say that the basis $\{s_\alpha\}$, or the corresponding point in $\mathcal B_k$, has $R$-bounded geometry if the metric $\tilde \omega$ does.

Given a basis $\{s_\alpha\}$ and a Hermitian matrix $A = (A_{\alpha \beta})$ define
$$
H_A = \sum A_{\alpha \beta} (s_\alpha, s_\beta)
$$
where we have taken the inner-product here using the pull-back of the Fubini--Study metric for which $\sum | s_\gamma|^2 =1$. Note that $H_A = \tr(A\mu)$ restricted to $X$ and so is the potential giving the infinitesimal deformation corresponding to $A$ of the restriction of the Fubini--Study metric to $X$. As a final piece of notation, we denote by $\| A \|_{\mathrm{op}}$ the maximum of the moduli of the eigenvalues of $A$ and by $\| A\|=\sqrt {\tr A^2}$ the norm of $A$ with respect to the Killing form. The first estimate we want is the following.

\begin{proposition}[Donaldson \cite{donaldson-scape1}]\label{donaldson's estimate}
There is a constant $C$ such that for all points of $\mathcal B_k$ with $R$-bounded geometry in $C^{r}$ and any Hermitian matrix $A$,
$$
\| H_A \|_{C^r} 
\leq C \| \bar \mu\|_{\mathrm{op}} \| A \|,
$$
where $\bar \mu=\int_X \mu\, \tilde\omega^n/n!$ is computed using the embedding corresponding to the point of $\mathcal B_k$ and the $C^r$-norm is taken with respect to the fixed reference metric $\tilde \omega_0$.
\end{proposition}

The key point is that $C$ depends only on $R$ and $r$, but not on $k$. This is proved more-or-less explicitly in the course of the proof of Lemma 24 of \cite{donaldson-scape1}, even though the end result is not stated in quite the form we give here. Accordingly we give only a sketch proof here, giving nearly word-for-word parts of the proof of Lemma 24 of \cite{donaldson-scape1}. 

\begin{proof}[Sketch of proof of Proposition \ref{donaldson's estimate}]
First we recall the following standard estimate. Let $Z$ be a compact complex Hermitian manifold, $E \to X$ a Hermitian holomorphic vector bundle and $P \subset Z$ a differentiable (real) submanifold. There is a constant $C$ such that for any $\sigma \in H^0(E,Z)$,
\begin{equation}\label{elliptic estimate}
\| \sigma \|_{C^r(P)} \leq C \| \sigma \|_{L^2(Z)}.
\end{equation}
Moreover, provided that the data $Z, P, E$ has bounded local geometry in $C^r$ in a suitable sense, $C$ can be taken to be independent of the particular manifolds and bundles involved.

We apply this to the manifold $Z = X \times \overline X$ where $\overline X$ is $X$ with the opposite complex structure. The Hermitian metric on $L^k$ induces a connecion in $L^k$; one component of the connection recovers the holomorphic structure on $L^k\to X$ whilst the other component makes $\overline L^k \to \overline X$ into a holomorphic line bundle. Let $E \to Z$ be the tensor product of the pull-back of $L^k$ from the first factor and $(\overline L^k)^*$ from the second. Given a Bergman metric in $\mathcal B_k$ we take the obvious induced K\"ahler metric on $Z$ and Hermitian metric in $E$. Let $P$ denote the diagonal in $Z$. We will use the estimate (\ref{elliptic estimate}) in this situation along with the fact that the constant can be taken independently of the Bergman metric used, provided it has $R$-bounded geometry in $C^r$.

 A holomorphic section $s$ of $L^k\to X$ defines a holomorphic section $\tilde s$ of $(\overline L^k)^* \to \overline X$ via the bundle isomorphism given by the fibre metric. Thus for any Hermitian matrix $A$ we get a holomorphic section
$$
\sigma_A = \sum A_{\alpha\beta}\, s_\alpha \otimes \tilde s_\beta
$$
of $E$ over $Z$. We have
$$
\| \sigma_A \|^2_{L^2(Z)}
=
\sum A_{\alpha \beta} \overline A_{\alpha' \beta'}
\langle s_\alpha, s_\alpha' \rangle 
\langle s_\beta, s_\beta' \rangle ,
$$
(where $\langle \cdot, \cdot \rangle$ denotes the $L^2$-inner-product). In matrix notation this reads 
$$
\| \sigma_A\|^2_{L^2(Z)}
=
\tr\left(A \bar\mu \bar\mu^* A^*\right).
$$
There is a standard inequality that for Hermitian matrices $G, F$, 
\begin{equation}\label{matrix inequality}
\tr (FGF) \leq \|F\|^2 \|G\|_\mathrm{op}
\end{equation} 
which here gives
$$
\| \sigma_A\|_{L^2(Z)} \leq \|\bar\mu\|_{\mathrm{op}} \|A\|
$$

Now, over $P$, the metric on $L^k$ defines a $C^\infty$ trivialisation of $E$ and the function $H_A = \sum A_{\alpha \beta}(s_\alpha, s_\beta)$ is just the restriction of $\sigma_A$ to the diagonal in this trivialisation. Hence, by the inequality (\ref{elliptic estimate}), we have 
$$
\| H_A\|_{C^r(X)} 
\leq 
C \|\sigma_A\|_{L^2(Z)} 
\leq
C \|\bar \mu\|_{\mathrm{op}} \|A\|.
$$
\end{proof}

We can rephrase this result by saying that under certain conditions, the Riemannian distance on $\mathcal B_k$ controls the $C^{r-2}$-norm on K\"ahler forms. To make this precise, let $\tilde\omega(s)$ for $s \in [0,1]$ denote a path of K\"ahler forms in $\mathcal B_k$. We denote by $L$ the length of the path, measured using \emph{large} symmetric Riemannian metric on $\mathcal B_k$, i.e., the metric corresponding to the Killing form $\tr A^2$. 

\begin{lemma}\label{controlling norm by distance}
If all the metrics $\tilde\omega(s)$ for $s \in [0,1]$ have $R$-bounded geometry in $C^{r}$ and also satisfy  $\|\bar\mu\|_\mathrm{op} < K$ then 
$$
\| \tilde \omega(0) - \tilde \omega(1) \|_{C^{r-2}} < CKL,
$$
where the $C^{r-2}$ norm is taken with respect to the reference metric $\tilde \omega_0$.
\end{lemma}
\begin{proof}
Let $A(s)$ denote the Hermitian matrix which is tangent to the given path. We have that 
$$
\left\|\frac{\del \tilde \omega}{\del t}\right\|_{C^{r-2}}
=
\left\|i \del\delb H_{A(s)}\right\|_{C^{r-2}}\\
\leq
CK\|A\|.
$$
The result now follows by integrating along the path.
\end{proof}

Of course, to apply this lemma we need to find regions in $\mathcal B_k$ which consist of $R$-bounded metrics and also for which $\|\bar\mu\|_\mathrm{op}$ is uniformly controlled. In this direction, we prove the following simple lemma. We denote by $d = k^{n+2}d_k$ the unscaled symmetric Riemannian metric on $\mathcal B_k$. 

\begin{lemma}\label{finding R bounded metrics}
Let $\tilde \omega_k \in \mathcal B_k$ be a sequence of metrics with $R/2$-bounded geometry in $C^{r+2}$ and such that $\|\bar \mu( \tilde \omega_k)\|_{\mathrm{op}}$ is uniformly bounded. Then there is a constant $C$ such that if $\tilde \omega \in \mathcal B_k$ has $d(\tilde\omega_k, \tilde \omega) < C$, then $\tilde \omega$ has $R$-bounded geometry in $C^{r}$.
\end{lemma}

\begin{proof}
There is a Hermitian matrix $B$ such that $\tilde \omega = e^B \cdot \tilde \omega_k$; note that $d(\tilde \omega_k, \tilde \omega) = \|B\|$. Let $s_\alpha$ be a basis for $H^0(L^k)$ defining the embedding corresponding to $\tilde \omega_k$ and chosen, moreover, so that $B = \mathrm{diag}(\lambda_\alpha)$ is diagonal in this basis. (This can be done thanks to $\U(N+1)$-invariance.) Then $\tilde \omega = \tilde \omega_k +i\delb\del v$ where
$$
e^{v} = \sum e^{2\lambda_\alpha} |s_\alpha|^2.
$$
Because $\tilde \omega_k$ has $R/2$-bounded geometry in $C^{r+2}$ and $\|\bar \mu (\tilde\omega_k) \|_{\mathrm{op}}$ uniformly bounded, Proposition \ref{donaldson's estimate} implies that there is a constant $c$ such that for any~$\alpha$, 
$$
\left\| |s_\alpha|^2 \right\|_{C^{r+2}} < c.
$$
(Apply the result to the matrix $A$ with a single 1 as entry $\alpha$ on the diagonal and zeros elsewhere.) It follows that 
$$
\|e^{v}\|_{C^{r+2}} \leq c e^{2\max |\lambda_\alpha|} \leq ce^{2\|B\|}.
$$
So a bound on $\|B\|$ gives uniform control of $\| \tilde\omega - \tilde \omega_k\|_{C^{r}}$. Hence there is a $C$ such that when $\|B\| = d(\tilde \omega, \tilde \omega_k)<C$, $\tilde \omega$ has $R$-bounded geometry in $C^{r}$. 
\end{proof}

\section{Projective estimates}\label{projective estimates}

Our final task is to control $\| \bar \mu\|_{\mathrm{op}}$. For the first lemma in this direction, we consider the situation from Tian's Theorem \ref{Tian}; so $h$ is a Hermitian metric in $L$ with positive curvature defining a K\"ahler form $\omega$ and $h_k \in \mathcal B_k$ is the sequence of Bergman metrics in $L$ corresponding to an $L^2(h^k, \omega)$-orthonormal basis of $H^0(L^k)$. 

\begin{lemma}\label{asymptotics of operator norm}
$\| \bar \mu(h_k) - 1_k \|_{\mathrm{op}} \to 0$, where $1_k\in i\un(N_k+1)$ is the identity matrix. Moreover, this convergence is uniform in $\omega$ in the sense that there is an integer $s$ such that if $\omega$ runs over a set of metrics bounded in $C^s$ and for which $\omega$ is bounded below, then the convergence is uniform.
\end{lemma}

\begin{proof}
We borrow another trick we learnt from \cite{donaldson-scape1}. Let $s_\alpha$ be a basis for $H^0(L^k)$ determining $\omega_k$. Given a continuous function $F \colon X \to \R$, set $A_F$ to be the Hermitian matrix with entries 
$$
(A_F)_{\alpha \beta}= \int_X  (s_\alpha, s_\beta) F\,
\frac{\omega^n}{n!}
$$
Then $\|A_F\|_\mathrm{op} \leq \|F\|_{C^0}$. This is because the map $A_F\colon H^0(L^k) \to H^0(L^k)$ factors through the space $V$ of all $L^2$-integrable sections as $A_F = \pi \circ M_F \circ j$ where $j \colon H^0(L^K) \to V$ is the inclusion, $M_F$ is multiplication by $F$ and $\pi \colon V \to H^0(L^k)$ is orthogonal projection in $V$.

We are interested in the matrix
\begin{eqnarray*}
\bar \mu_{\alpha \beta} 
&=& 
\int_X 
	\frac{(s_\alpha, s_\beta)}{\rho_k}
	\frac{\omegaFS^n}{n!}\\
&=&
\int_X (s_\alpha, s_\beta)(1+O(k^{-1}))
	\frac{\omega^n}{n!}
\end{eqnarray*}
where we have used the asymptotic expansion of $\rho_k$ and $\omega = \frac{1}{k}\omegaFS + O(k^{-2})$. So $\bar\mu - 1_k$ is the matrix associated to a function $F_k \colon X \to \C$ with $\| F_k\|_{C^0} = O(k^{-1})$. Hence $\| \bar\mu - 1_k\|_\mathrm{op} = O(k^{-1})$. The convergence is uniform in $\omega$ because the asymptotic expansion of $\rho_k$ is.
\end{proof}

\begin{remark}\label{diagonal}
As with the asymptotics of $\rho_k$ and $Q_k$, the fact the convergence of Lemma \ref{asymptotics of operator norm} is uniform allows us to pass from a single metric $\omega$ to a sequence $\omega(k)$ which converges in $C^\infty$. If we denote by $\omega_k(k) \in \mathcal B_k$ the $k^{\mathrm{th}}$ standard projective approximation to $\omega(k)$ it follows from the uniformity that $\|\bar \mu(\omega_k(k)) - 1_k \|_\mathrm{op}\to 0$
\end{remark}

The remainder of this section is devoted to controlling $\|\bar \mu\|_\mathrm{op}$ in terms of the Riemannian distance in the Bergman space. Our arguments will apply simultaneously to all $\mathcal B_k$ without $k$ playing a role. Accordingly, until the end of the section we drop the reference to $k$ and work on the Bergman space $\mathcal B \cong \GL(N+1)/\U(N+1)$ associated to a given subvariety $X \subset\C\P^N$.

Given a point $b\in \mathcal B$ and tangent vector $A\in T_b\mathcal B \cong i\un(N+1)$, we differentiate  $\bar\mu \colon \mathcal B \to i\un(N+1)$ at $b$ to obtain $\diff\bar\mu(A) \in i\un(N+1)$. The first fact we need---which appears, for example, in \cite{phong-sturm-sefakem}---is the relationship between $\diff \bar \mu(A)$ and the extrinsic geometry of the embedding $X \subset \C\P^N$ corresponding to $b \in \mathcal B$. Let $\xi_A$ denote the vector field on $\C\P^N$ corresponding to $A$. Let $\xi_A^{TX}$ denote the component of $\xi_A |_X$ which is tangent to $X$ and $\xi_A^\perp$ the component which is perpendicular. Finally let $(\cdot, \cdot)$ denote the Fubini--Study inner-product on tangent vectors. 

\begin{lemma}\label{second derivative}
For any pair of Hermitian matrices $A, B \in i\un(N+1)$,
$$
\tr(B\, \diff \bar\mu(A))
=
\int_X (\xi_A^\perp, \xi_B^\perp)\, \frac{\omegaFS^n}{n!}.
$$
\end{lemma}

\begin{proof}
\begin{eqnarray*}
\tr(B\, \diff \bar\mu(A)) 
	&=&
		\int_X \tr (B\, \diff \mu(A))\, \frac{\omegaFS^n}{n!}
		+
		\int_X \tr(B \mu ) \frac{L_{\xi_A} (\omegaFS^n)}{n!},\\
	&=&
		\int_X \Big((\xi_A, \xi_B) -H_B \Delta H_A\Big)
			\frac{\omegaFS^n}{n!},\\
	&=&
		\int_X\Big((\xi_A, \xi_B) - (\xi^{TX}_A, \xi^{TX}_B)\Big)
			\frac{\omegaFS^n}{n!}.
\end{eqnarray*}
Here the various equalities all follow from the fact that $-i\mu$ is a moment map for the $\U(N+1)$-action on $\C\P^N$; we have
\begin{eqnarray*}
\tr (B \, \diff \mu (A) ) 
	&=& \omegaFS(J\xi_A, \xi_B),\\
	&=& (\xi_A, \xi_B).\\
L_{\xi_A}\omegaFS 
	&=& 2i\delb\del \left(\tr A\mu\right),\\
	&=& 2i \delb \del H_A.\\
\left(L_{\xi_A}\omegaFS^n\right)|_{X}
	&=& - \Delta H_A \left(\omegaFS^n|_X\right).
\end{eqnarray*}
\end{proof}

We now continue with a series of identities and estimates in projective geometry which provide the pieces needed to control $\| \bar \mu\|_\mathrm{op}$.

\begin{lemma}\label{pointwise equality}
Let $A,B \in i\un(N+1)$ be Hermitian matrices. At every point of $\C\P^N$,
$$
H_A H_B + ( \xi_A, \xi_B ) = \tr (AB\mu).
$$
\end{lemma}
\begin{proof}
By $\U(N+1)$ equivariance, it suffices to consider the point $p = [1\colon 0\colon \cdots \colon 0]$. Let $(x_1, \ldots, x_N) \mapsto [1\colon x_1 \colon \cdots \colon x_N]$ be unitary coord\-inates. At $p$, $\mu(p)$ has a single non-zero entry which is a one in the top left corner. Hence $H_A(p) = A_{00}$, $H_B(p) = B_{00}$ and 
$$
\tr (AB\mu) = A_{00}B_{00}+A_{01}B_{10} + \cdots + A_{0N}B_{N0}.
$$
Meanwhile, at $p$, the coordinate vectors $\del_i$ are orthonormal whilst $\xi_A = A_{01} \del_1 + \cdots + A_{0N} \del_N$ and similarly for $\xi_B$. Putting the pieces together and using $B^*=B$ gives the result.
\end{proof}

\begin{remark}
As an aside, it is interesting to compare this result with the analytic estimate in Proposition \ref{donaldson's estimate}. It follows from Lemma \ref{pointwise equality} that for any $A$, at every point of $X$,
\begin{eqnarray*}
H_A^2 + |\nabla H_A|^2
	&=&
		H_A^2 + |\xi_A^{TX}|^2,\\
	&\leq&
		H_A^2 + |\xi_A|^2,\\
	&=&
		\tr (A^2\mu),\\
	&\leq&
		\|A\|^2.
\end{eqnarray*}
So the $C^1$ case of Proposition \ref{donaldson's estimate}, $\|H_A\|_{C^1} \leq \|A\|$, comes ``for free'' from projective geometry with no need to use analysis (and with no need to involve $\| \bar\mu\|_{\mathrm{op}}$ in the bound). 
\end{remark}

\begin{lemma}\label{L21 equality}
For any Hermitian matrices $A, B\in i \un(N+1)$, 
$$
\tr(B \, \diff \bar\mu(A)) + \langle H_A, H_B \rangle_{L^2_1(X)}
=
\tr(AB\bar \mu).
$$ 
\end{lemma}
\begin{proof}
From Lemma \ref{pointwise equality} we have, at every point of $X$,
$$
H_A H_B + (\xi_A^{TX}, \xi_B^{TX}) + (\xi_A^\perp, \xi_B^\perp) = \tr (AB\mu)
$$
Now use the identity $\xi_A^{TX} = \nabla H_A$, integrate over $X$ and apply Lemma \ref{second derivative}.
\end{proof}

\begin{remark}
This identity fits into Donaldson's ``double quotient'' picture described in \S 2.1 of \cite{donaldson-scape1}; we explain this here, although we make no direct use of this observation later. Donaldson considers the infinite dimensional space $\mathcal X = \Gamma (L^k)^{N_k+1}$ (where $\Gamma$ denotes \emph{smooth} sections). Given a Hermitian metric $h$ in $L$ with positive curvature $2\pi i \omega$, $\mathcal X$ is formally a K\"ahler manifold, where the Riemannian metric is given by the $L^2(h^k, \omega)$-inner-product on sections of $L^k$ along and the complex structure is given by multiplication of sections by $i$. Two groups act on $\mathcal X$, one finite-dimensional the other infinite-dimensional. The finite dimensional  group is $\GL(N_k+1)$ which acts on $\mathcal X$ by K\"ahler isometries, mixing the $\Gamma(L^k)$ factors in the obvious way. The infinite dimensional group is the group $\mathcal G$ of Hermitian bundle maps $L \to L$ which preserve the Chern connection of $h$; this acts preserving the $L^2$-inner-product on $\Gamma (L^k)$ and hence by it acts by K\"ahler isometries on $\mathcal X$. Whilst the complexification of $\mathcal G$ doesn't exist, one can still make sense of the complex ``orbits'' in $\mathcal X$. 

Assume that $(h, \omega)$ come from a projective embedding defined via a basis $\underline s = (s_0, \ldots s_{N_k}) \in \mathcal X$. Given a Hermitian matrix $A \in i\un(N_k+1)$, we get an infinitesimal change in the basis $\underline s$, i.e., a tangent vector $V_A \in T_{\underline s} \mathcal X$. Let $ P \subset T_{\underline s} \mathcal X$ denote the tangent space to the complex ``orbit'' through $\underline s$. We can decompose $V_A$ into two components, the part $V'_A$ which is in $P$ and the part $V''_A$ which is orthogonal. Doing likewise for a second Hermitian matrix $B$, we have the obvious identity:
$$
(V_A, V_B) = (V'_A, V'_B) + (V''_A, V''_B).
$$

Proposition 19 in \cite{donaldson-scape1} gives $V'_A$ explicitly in terms of $H_A$ and using this one can write out the terms in this identity giving
\begin{eqnarray*}
(V_A, V_B) &=& \tr(AB\bar \mu),\\
(V'_A, V'_B) &=& \langle H_A, H_B \rangle_{L^2_1(X)},\\
(V''_A, V_B'') &=& \tr(B \, \diff \bar\mu(A)).
\end{eqnarray*}
So Lemma \ref{L21 equality} amounts to the orthogonal decomposition $T_{\underline s}\mathcal X = P \oplus P^\perp$. Meanwhile, the equality $(V''_A, V''_B) = \tr (B \, \diff \bar\mu(A))$ is a consequence of the fact that balancing flow is the downward gradient flow of the Kempf--Ness function associated to the finite dimensional moment map problem that remains after taking the symplectic reduction by  the action of $\mathcal G$.
\end{remark}

\begin{lemma}\label{L21 estimate}
For any Hermitian matrix $A \in i\un(N+1)$, 
$$
\| H_A \|^2_{L^2_1} \leq \|A\|^2 \|\bar\mu\|_{\mathrm{op}}.
$$
\end{lemma}

\begin{proof}
It follows from  Lemma \ref{L21 equality}  that 
$$
\|H_A\|^2_{L^2_1} = \tr(A^2 \bar\mu) - \tr (A \, \diff\bar\mu(A)).
$$
From Lemma \ref{second derivative},
$$
\tr (A \, \diff\bar\mu(A)) = \int_X |\xi_A^\perp|^2\frac{\omegaFS^n}{n!} >0.
$$ 
Hence
$$
\|H_A\|^2_{L^2_1} \leq \tr(A^2 \bar\mu) \leq \|A\|^2 \|\bar\mu\|_{\mathrm{op}},
$$
where the second inequality follows from inequality (\ref{matrix inequality}).
\end{proof}
 
\begin{lemma}\label{infinitesimal control}
For any Hermitian matrix $A \in i\un(N+1)$,
$$
\left\| \diff \bar\mu(A)\right\|
\leq 
2 \|A\| \|\bar\mu\|_\mathrm{op}.
$$
\end{lemma}
\begin{proof}
From Lemma \ref{L21 equality} with $B = \diff \bar\mu(A)$, 
$$
\left\| \diff \bar\mu(A)\right\|^2
=
\tr\left( \diff\bar\mu(A)^2\right)
=
\tr (A\, \diff\bar\mu(A)\bar\mu)
-
\langle H_A, H_{\diff \bar\mu(A)} \rangle _{L^2_1}.
$$
Now apply Cauchy--Schwarz, Lemma \ref{L21 estimate} and inequality (\ref{matrix inequality}) to deduce
$$
\|\diff \bar\mu(A)\|^2
\leq
2\| A\| \|\diff\bar\mu(A)\| \|\bar\mu\|_{\mathrm{op}}.
$$
\end{proof}

Finally, we are in a position to control $\|\bar\mu\|_\mathrm{op}$ in terms of Riemannian distance on $\mathcal B$.

\begin{proposition}\label{global control}
Let $b_0 , b_1 \in \mathcal B$ and let $d(b_0, b_1)$ denote the Riemannian distance between $b_0$ and $b_1$. Then
$$
\left\| \bar\mu(b_1)\right\|_{\mathrm{op}}
\leq
e^{2d(b_0, b_1)} \left\| \bar\mu(b_0)\right\|_{\mathrm{op}}.
$$
\end{proposition}

\begin{proof}
Let $A$ generate the geodesic $e^{tA}$ in $\mathcal B$ joining $b_0$ and $b_1$ so that $\|A\| = d(b_0, b_1)$. As we run along the geodesic from $b_0$ to $b_1$, the rate of change of $\|\bar\mu\|_\mathrm{op}$ is at most 
$$
\|\diff\bar\mu(A)\|_\mathrm{op} 
\leq
\| \diff \bar\mu (A) \|
\leq 
2\| A\| \| \bar\mu\|_{\mathrm{op}},
$$
by Lemma \ref{infinitesimal control} and so the growth is sub-exponential.
\end{proof}

\section{Completing the proof of Theorem \ref{main result 2}}\label{completing the proof} 
 
Now, finally, all the pieces are in place to prove our main result. 
We begin by recalling our notation. Let $h$ be a Hermitian metric in $L$ with positive curvature $2\pi i \omega$; denote by $\omega(t)$ the Calabi flow starting at $\omega$. Let $\iota_k$ be the embedding of $X$ defined by a basis of $H^0(L^k)$ which is orthonormal with respect to the $L^2$ inner product defined by $h(0)$ and $\omega(0)$. Let $\omega_k = \frac{1}{k}\iota_k^*\omegaFS$ denote the standard sequence of projective approximations to $\omega(0)$. let $\iota_k(t)$ solve the balancing flow (\ref{balancing flow}) and let $\omega_k(t)= \frac{1}{k}\iota_k(t)^*\omegaFS$. We must prove first that $\omega_k(t) \to \omega(t)$ in $C^\infty$.

Recall that in Theorem \ref{higher order approximation}, for any given integer $m$ we constructed a sequence of flows $\omega(k;t)$ which satisfies:
\begin{enumerate}
\item 
$\omega(k;t) \to \omega(t)$ in $C^\infty$ as $k \to \infty$;
\item
$d_k (h_k(t), h'_k(t) ) \leq Ck^{-m-1}$ where $h_k(t)$ is the balancing flow, $h'_k(t)$ denotes the $k^\mathrm{th}$ standard projective approximation to $\omega(k;t)$ and $d_k$ is the scaled Riemannian distance function on $\mathcal B_k$ corresponding to the Killing form $k^{-n-2} \tr A^2$ on Hermitian matrices.
\end{enumerate}

Let $\omega'_k(t) \in c_1(L)$ denote the (rescaled) K\"ahler form corresponding to $h'_k(t)$. It follows from point 1 above and the uniformity of the asymptotic expansion of $\rho_k$ that $\omega'_k(t) \to \omega(t)$ in $C^\infty$. We will show $\omega_k(t) \to \omega(t)$ by proving that 
$$
\|\omega_k(t) - \omega'_k(t) \|_{C^r(\omega(t))} \to 0
$$ 
as $k \to \infty$. To do this we will control the $C^r$ norm on metrics by the distance $d_k$ in $\mathcal B_k$ along the geodesics joining $h'_k(t)$ to $h_k(t)$. We can do this by Lemma~\ref{controlling norm by distance} provided all the points have $R$-bounded geometry in $C^{r+2}$ and all have $\|\bar\mu\|_{\mathrm{op}}$ uniformly controlled as well. 

We begin with the control of $\|\bar\mu\|_\mathrm{op}$. By Lemma \ref{asymptotics of operator norm} and Remark \ref{diagonal}, $\| \bar\mu\|_\mathrm{op}$ is controlled for $h'_k(t)$ uniformly in $k$. Now we apply Proposition \ref{global control}, for which we need $h_k(t)$ to be a uniformly bounded distance from $h'_k(t)$ in the \emph{unscaled} distance $d = k^{n+2} d_k$. Provided we take $m \geq n+1$ this holds giving that $\| \bar\mu \|_\mathrm{op}$ is uniformly bounded along the geodesics joining $h'_k(t)$ to $h_k(t)$.

Next, we establish that the points of these geodesics have $R$-bounded geometry in $C^{r+2}$. We use $\omega(t)$ as our reference metric and apply Lemma \ref{finding R bounded metrics} to the sequence $\omega'_k(t)$ (but with $r$ replaced by $r+2$). We have already observed that the part of the hypothesis concerning $\| \bar\mu\|_\mathrm{op}$ is satisfied and the metrics $\omega'_k(t)$ certainly have $R/2$-bounded geometry in $C^{r+4}$, since they converge in $C^\infty$ to  $\omega(t)$. Now, provided we take $m \geq n+2$, the \emph{unscaled} distance $d (h'_k(t), h_k(t))$ tends to zero and so, for sufficiently large $k$, all points on the geodesics joining $h'_k(t)$ and $h_k(t)$ have $R$-bounded geometry in $C^{r+2}$.

Finally, we can apply Corollary \ref{controlling norm by distance}. This tells us that for the \emph{unscaled} metrics there is some constant $M$ such that
$$
\|k \omega_k(t) - k\omega'_k(t) \|_{C^r(k\omega(t))}
\leq
M k^{n+2}d_k(h_k(t),h'_k(t))
\leq
MC k^{n+1 - m}
$$
Rescaling this inequality we see that
$$
\| \omega_k(t) - \omega'_k(t) \|_{C^r(\omega(t))}
\leq
MC k^{(r/2) + 1+ n - m}.
$$
So provided we take $m > \frac{r}{2} + 1 +n$ we obtain that $\omega_k(t)$ converges to $\omega(t)$ in $C^r$.

For $t\in [0,T]$, $\{\omega(t)\}$  is a compact set of metrics; from here it is easy to check that the convergence $\omega_k(t) \to \omega(t)$ is uniform in $t$. There are various places where uniformity must be checked. Firstly, we have used $\omega(t)$ to define the $C^r$-norms, but the compactness ensures all these norms are uniformly equivalent. Secondly, we have applied asymptotic expansions for $\rho_k$ and $Q_k$, but these are both uniform as $t$ varies, by uniformity in the relevant Theorems \ref{asymptotics of Bergman}, \ref{Liu-Ma} and \ref{converge in Cinfty}. We also must check that $\omega(k;t)$ converges uniformly to $\omega(t)$. This holds since, for given $r$, there are only finitely many perturbations $\eta_j$ present. Finally, if we denote by $\dot \phi_k$ the potential for the $t$-derivative of $\omega_k(t)$, it follows from Theorem \ref{limit of balancing potentials} that $\dot \phi_k(t) \to S(\omega(t)) - \bar S$ in $C^\infty$. This is uniform for $t$ in a compact interval, again by the uniformity of the asymptotics in Theorems \ref{asymptotics of Bergman}, \ref{Liu-Ma} and \ref{converge in Cinfty}.

\appendix

\section*{Appendix: Asymptotics of the operators $Q_k$
\footnote{Written by Kefeng Liu and Xiaonan Ma}}
\label{appendix}

This note is a continuation of \cite{liu-ma-arosnricdg} providing a technical result needed in the preceding article of Fine. We refer to \cite{donaldson-snricdg}, \cite{liu-ma-arosnricdg} and Fine's paper for the context of the problem. We also refer the readers
to the recent book \cite{MM06a} for more information on the Bergman kernel.

We begin by recalling the basic setting and notation in \cite{liu-ma-arosnricdg},
which we will use freely throughout.

Let $(X,\omega, J)$ be a compact K{\"a}hler manifold with
$\dim_{\C}X=n$, and let $(L, h^L)$ be a holomorphic Hermitian line
bundle on $X$. Let $\nabla ^L$ be the holomorphic Hermitian
connection on $(L,h^L)$ with curvature $R^L$. We assume that
 \begin{align} \label{n1}
\frac{\sqrt{-1}}{2 \pi} R^L=\omega.
\end{align}

Let $g^{TX}(\cdot,\cdot):= \omega(\cdot,J\cdot)$ be the Riemannian metric on $TX$
induced by $\omega, J$. Let $dv_X$ be the Riemannian volume
form of $(TX, g^{TX})$, then $dv_X= \omega^n/n!$.
Let $d\nu$ be any volume form on $X$.
Let $\eta$ be the positive function on $X$ defined by
 \begin{align} \label{n3}
dv_X= \eta\, d\nu.
\end{align}
The $L^2$--scalar product $\langle \quad\rangle_\nu$ on $C^\infty (X,L^p)$,
the space of smooth sections of $L^p$, is given by
\begin{align}\label{n2}
&\langle \sigma_1,\sigma_2 \rangle_\nu
:=\int_X\langle \sigma_1(x),\sigma_2(x)\rangle_{h^{L^p}}\,d\nu(x)\,.
\end{align}

Let $P_{\nu, p}(x,x')$ $(x,x'\in X)$ be the smooth kernel of
the orthogonal projection from
$(C ^\infty(X, L^p),\langle\quad \rangle_\nu )$
onto $H^0(X,L^p)$, the space of the holomorphic sections of $L^p$ on $X$,
 with respect to $d\nu (x')$.
Following \cite[\S 4]{donaldson-snricdg}, set
\begin{align}\label{1n1}
K_p(x,x'):= |P_{\nu, p}(x,x')|^2_{h^{L^p}_x\otimes h^{L^{p*}}_{x'}}, \quad
R_p:= (\dim H^0(X, L^p))/ {\rm Vol}(X,\nu),
\end{align}
here ${\rm Vol}(X,\nu):= \int_X d\nu$. Set ${\rm Vol}(X,dv_X):= \int_X dv_X$.

Let $Q_{K_p}$ be the integral operator associated to $K_p$
which is defined for $f\in C^\infty (X)$,
\begin{equation}\label{1n2}
Q_{K_p} (f)(x):= \frac{1}{R_p} \int_X K_p(x,y) f(y)d\nu(y).
\end{equation}

Let $\Delta$ be the (positive) Laplace operator on $(X, g^{TX})$
acting on the functions on $X$.
We denote by $|\quad|_{L^2}$ the $L^2$-norm on the function
on $X$ with respect to $dv_X$.

The following result is an improvement of \cite[Theorem 0.1]{liu-ma-arosnricdg},
 where it is proved for $q=0$ and $q=1$.
\begin{theorem} \label{nt1} For and $q\in \N$, 
there exists a constant $C>0$ such that for any
$f\in C^\infty (X)$, $p\in \N$,
\begin{align}\label{1n4}
\begin{split}
&  \left|\left( \Big(\frac{\Delta}{p}\Big)^q Q_{K_p}
- \frac{{\rm Vol}(X, \nu)}{{\rm Vol}(X, dv_X)}
\Big(\frac{\Delta}{p}\Big)^q\eta \exp\Big(-\frac{\Delta}{4\pi p}\Big)\right) f\right|_{L^2}
\leqslant \frac{C}{p}    \left|f\right|_{L^2}.
\end{split}
\end{align}
Moreover, \eqref{1n4} is uniform in that there is an integer $s$
such that if all data $h^L$, $d\nu$ run over a set which is
bounded in $C^s$-topology and that $g^{TX}$, $dv_X$ are bounded from
below, then the constant $C$ is independent of $h^L$, $d\nu$.
\end{theorem}
\begin{proof}
Let $E=\C$ be the trivial holomorphic line bundle on $X$. Let
$h^E$ the metric on $E$ defined by $|\mathbf{e}|_{h^E}^2= 1$, 
here $\mathbf{e}$ is the canonical unity element of $E$. 
We identify canonically $L^p$ to $L^p\otimes E$ by section $\mathbf{e}$.

Let $h^E_\omega$ be the metric on $E$ defined by 
$|\mathbf{e}|_{h^E_\omega}^2= \eta ^{-1}$.
Let $\langle\quad \rangle_\omega$ be the Hermitian product on
$C^\infty (X, L^p\otimes E)= C^\infty (X, L^p)$
induced by $h^L, h^E_\omega$, $dv_X$ as in \eqref{n2}.
If $P_{\omega,p}(x,x')$, ($x,x'\in X$) denotes the smooth kernel of
the orthogonal projection $P_{\omega,p}$ from
$(C^\infty (X, L^p\otimes E), \left\langle\,\cdot,\cdot \right \rangle_\omega)$
onto $H^0(X, L^p\otimes E)= H^0(X, L^p)$ with respect to $dv_{X}(x)$. 
By \cite[(11)]{liu-ma-arosnricdg}, we have
\begin{equation}\label{n5}
P_{\nu, p}(x,x')= \eta(x')\, P_{\omega,p}(x,x').
\end{equation}
For $f\in C^\infty (X)$, set
 \begin{align}\label{n7}
\begin{split}
&K_{\omega,p}(x,x')= |P_{\omega, p}(x,x')|^2_{(h^{L^p}\otimes h^E_\omega)_x\otimes
(h^{L^{p*}}\otimes h^{E^*}_\omega)_{x'}},\\
 &(K_{\omega,p}f)(x)=  \int_X  K_{\omega,p}(x,y)f(y) dv_X(y).
\end{split}
\end{align}
Then by  \cite[(15)]{liu-ma-arosnricdg}, we have
\begin{equation}\label{n9}
Q_{K_p} (f)(x)
= \frac{1}{R_p} \int_X  K_{\omega,p}(x,y) \eta (x)f(y) dv_X(y).
\end{equation}

Now we use the normal coordinate as in \cite{liu-ma-arosnricdg}.
Then under our identification, $P_{\omega,p}(Z,Z')$ is a function on
$Z,Z'\in T_{x_0}X$, $|Z|,|Z'|\leqslant \varepsilon$,
we denote it by $P_{\omega,p,x_0}(Z,Z')$ with complex values.

Note that
$|P_{\omega,p,x_0}(Z,Z')|^2 =P_{\omega,p,x_0}(Z,Z')\overline{P_{\omega,p,x_0}(Z,Z')}$,
thus from \cite[(19),(20)]{liu-ma-arosnricdg},  \eqref{n7}, 
there exist  $J^\prime_{r}(Z,Z')$ polynomials
in $Z,Z'$ such that \footnote{There is a misprint in \cite[(25)]{liu-ma-arosnricdg},
we need to move a factor $\frac{1}{p^{2n}}$ into the parenthesis,
 thus $\frac{1}{p^{2n+1}} \Delta_Z\Big( K_{\omega,p,x_0}\cdots $
 therein should be read as $\frac{1}{p} \Delta_Z
\Big(\frac{1}{p^{2n}} K_{\omega,p,x_0}\cdots $}
\begin{multline}\label{n15}
\left |\frac{1}{p^{q}} \Delta_Z^q
\Big(\frac{1}{p^{2n}} K_{\omega,p,x_0}(Z,Z')\right.\\
\left. - 
\Big(1+  \sum_{r=2}^k \frac{1}{p^{r/2}}  J^\prime_{r} (\sqrt{p} Z,\sqrt{p} Z')\Big)
e ^{-\pi p |Z-Z^\prime|^2} \Big)\right|_{C^0(X)}\\
\leqslant C p^{-(k+1)/2}
(1+|\sqrt{p} Z|+|\sqrt{p} Z'|)^N
\exp (- C_0 \sqrt{p} |Z-Z'|)+ O(p^{-\infty}).
\end{multline}

For a function $f\in C^\infty(X)$, we denote it as $f_{x_0}(Z)$
a family (with parameter $x_0$) of function of $Z$ in the normal coordinate
near $x_0$.

Observe that in the normal coordinate, 
we denote by $g_{ij}(Z)= \left\langle e_i,e_j\right\rangle_Z$
with $e_i=\tfrac{\partial} {\partial Z_i}$, 
and let $(g^{ij}(Z))$ be the inverse of
the matrix $(g_{ij}(Z))$.
If $\Gamma _{ij}^l$ is the connection form of $\nabla ^{TX}$
with respect to the basis $\{e_i\}$, then we have $(\nabla
^{TX}_{e_i}e_j)(Z) = \Gamma _{ij}^l (Z) e_l$. 
Set $\Delta_0= -\sum_{j=1}^{2n} \tfrac{\partial^2} {\partial Z_j^2}$,
then  
\begin{align} \label{1n16}
\Delta_Z=\Delta_0 -\sum_{i, j=1}^{2n}
 \Big( (g^{ij}(Z)-1) \tfrac{\partial^2} {\partial Z_i \partial Z_j} 
- g^{ij}(Z) \Gamma_{ij}^k (Z)\tfrac{\partial} {\partial Z_k}\Big).
\end{align}
As $g^{ij}(Z)= 1+ O(|Z|^2), \Gamma_{ij}^k (Z)= O(|Z|)$
(cf. \cite[(1.2.19), (4.1.102)]{MM06a}), 
by recurrence, we know 
\begin{equation} \label{n19}
\Delta_Z^q e^{-\pi p |Z-Z^\prime|^2}
= \Delta_0^q  e^{-\pi p |Z-Z^\prime|^2} 
+\sum_{i=0}^2 h_i( Z,\sqrt{p},\sqrt{p}Z,\sqrt{p}Z^\prime)e^{-\pi p |Z-Z^\prime|^2}.
\end{equation}
Here $h_i(Z,a,x,y)$ are polynomials on $a,x,y$, and the degree on $a$
 is $\leqslant 2q-2+i$, moreover, the coefficients of 
 $h_i(Z,a,x,y)$ as a function on $Z$ is $C^\infty$ and $ O(|Z|^i)$.
Thus 
\begin{align} \label{1n19}
p^{-q} \Delta_Z^q e^{-\pi p |Z-Z^\prime|^2}
= p^{-q} (\Delta_0^q e^{-\pi p |Z-Z^\prime|^2})|_{Z=0}
+ p^{-q} h(\sqrt{p},\sqrt{p}Z^\prime) e^{-\pi p |Z^\prime|^2},
\end{align}
and $h(a,Z^\prime)$ is a polynomial on $a$ and $Z^\prime$,
and its degree on $a$ is $\leqslant 2q-2$.

{}From \eqref{n15}, \eqref{1n19} and \cite[(27)]{liu-ma-arosnricdg},
\begin{multline}\label{n21}
\left|p^{-n-q} \Delta^q K_{\omega,p} f
- p^{n-q} \int_{|Z^\prime|\leqslant \varepsilon} (\Delta_0^q e^{-\pi p |Z-Z^\prime|^2})|_{Z=0}  f_{x_0}(Z^\prime)
dv_X(Z^\prime)\right|_{L^2} \\
\leqslant \frac{C}{p}  \left|f\right|_{L^2}.
\end{multline}

Let $e ^{-u\Delta}(x,x')$ be the smooth kernel of the heat operator
$e^{-u \Delta}$ with respect to $dv_X(x')$.
By \cite[(35)]{liu-ma-arosnricdg}, 
there exist $\phi_{i,x_0}(Z^\prime)$ such that uniformly
for $x_0\in X$, $Z^\prime\in T_{x_0}X, |Z^\prime|\leqslant \varepsilon$,
we have the following asymptotic expansion when $u\to 0$,
\begin{multline}\label{n22}
\left| \frac{\partial ^l}{\partial u^l} \Big(e^{-u\Delta}(0, Z')
- (4\pi u)^{-n} \Big(1+ \sum_{i=1}^k u^i \phi_{i,x_0}(Z')\Big)
e^{-\frac{1}{4u} |Z^\prime|^2}\Big)\right|_{C^0(X)}\\ 
=O(u^{k-n-l+1}).
\end{multline}
 
Observe that
\begin{eqnarray}
\nonumber
\Delta^q \exp\Big(-\frac{\Delta}{4\pi p}\Big)
&=&  
(-1)^{q} (\tfrac{\partial}{\partial u} e^{-u\Delta})
\Big|_{u= \frac{1}{4\pi p}},\\
\nonumber
p^n (\Delta_0^q e^{-\pi p |Z-Z^\prime|^2})|_{Z=0}
&=& 
\Big(\Delta_0^q \exp\Big(-\frac{\Delta_0}{4\pi p}\Big)\Big)(0,Z^\prime)\\
\label{n24}
&=&
(-1)^q(\tfrac{\partial^q}{\partial u^q} e^{-u\Delta_0})
\Big|_{u= \frac{1}{4\pi p}} (0,Z^\prime).
\end{eqnarray}

By \eqref{n21}, \eqref{n22}, \eqref{n24} and \cite[(27)]{liu-ma-arosnricdg}, we have
\begin{align}\label{1n20}
\left| p^{-q}\Big(p^{-n}\Delta^q  K_{\omega,p}
-\Delta^q \exp\Big(-\frac{\Delta}{4\pi p}\Big)\Big) f \right|_{L^2}
\leqslant \frac{C}{p}  \left|f\right|_{L^2}.
\end{align}
Thus we have proved \eqref{1n4} when $\eta=1$.

If $\eta\neq 1$, set 
\begin{align}\label{1n21}
\begin{split}
&K_{\eta, \omega,p,q}(x,y)
= \langle d\eta (x), d_x \Delta^{q-1}_x K_{\omega,p}(x,y) \rangle_{g^{T^*X}},\\
&(K_{\eta, \omega,p,q}f)(x) = \int_X K_{\eta, \omega,p,q}(x,y) f(y) dv_X(y).
\end{split}
\end{align}
Then from \cite[(19),(20),(27)]{liu-ma-arosnricdg},
\eqref{n19} and \eqref{1n21}, we get
\begin{multline}\label{1n23}
\left|p^{-n-q} K_{\eta, \omega,p,q}f \right.\\
\left. - p^{n-q}  \int_{|Z^\prime|\leqslant \varepsilon}
\sum_{i=1}^{2n} (\tfrac{\partial}{\partial Z_i}\eta)(x_0,0) 
(\tfrac{\partial}{\partial Z_i} \Delta_0^{q-1}
 e^{-\pi p |Z-Z^\prime|^2})|_{Z=0}f_{x_0}(Z^\prime)
dv_X(Z^\prime)\right|_{L^2}\\
\leqslant \frac{C}{p} \left|f\right|_{L^2},
\end{multline}
here $C$ is independent on $p$.

By \cite[(33)]{liu-ma-arosnricdg}, we get the analogue of \cite[(36)]{liu-ma-arosnricdg}
\begin{multline}\label{1n24}
\Big|\frac{\partial ^l}{\partial u^l} \Big[
\langle  d\eta (x_0), d_{x_0}e^{-u\Delta} 
\rangle_{g^{T^*X}}(0,Z^\prime)\\
-  (4\pi u)^{-n} 
\sum_{i=1}^{2n} (\frac{\partial}{\partial Z_i}\eta)(x_0,0)\frac{Z_i^\prime}{2u}
 \Big(1+ \sum_{i=1}^k u^i \phi_{i,x_0}(Z')\Big)\Big) 
e^{-\frac{1}{4u} |Z^\prime|^2}\\
-  (4\pi u)^{-n} \sum_{i=1}^k u^i
\langle d\eta (x_0), (d_{x_0}\Phi_i)(0,Z^\prime)\rangle
 e^{-\frac{1}{4u} |Z^\prime|^2} \Big]
 \Big|_{C^0(X)} = O(u^{k-n-l+\frac{1}{2}}).
\end{multline}

From  \eqref{1n23}, \eqref{1n24} and \cite[(27)]{liu-ma-arosnricdg}, 
\begin{align}\label{n25}
&\left|p^{-q} \Big(p^{-n}K_{\eta,\omega,p,q} 
- \langle d\eta, d \Delta^{q-1} \exp({-\frac{\Delta}{4\pi p}})
 \rangle \Big)f \right|_{L^2}
\leqslant \frac{C}{p}  \left|f\right|_{L^2}.
\end{align}

Finally  
\begin{multline}\label{n26}
(\Delta^q (\eta K_{\omega,p}))(x,y)
=\eta(x) \Delta_x^q K_{\omega,p}(x,y)\\
-2 \langle d\eta (x), d_x\Delta_x^{q-1} K_{\omega,p}(x,y) \rangle_{g^{T^*X}}
+ \widetilde{K}_{\omega,p},
\end{multline}
here $\widetilde{K}_{\omega,p}$ has $\leqslant 2q-2$ derivative on 
 $K_{\omega,p}(x,y)$, thus 
\begin{align}\label{n29}
\left|\widetilde{K}_{\omega,p} f \right|_{L^2}
\leqslant C\left(\left|K_{\omega,p} f \right|_{L^2}
 +  \left|\Delta_x^{q-1} K_{\omega,p} f \right|_{L^2}\right).
\end{align}

Note also
$R_p= \frac{{\rm Vol}(X, dv_X)}{{\rm Vol}(X, \nu)} p^n + O(p^{n-1})$.
 From \eqref{n9}, \eqref{1n20}, \eqref{n25}-\eqref{n29}, we get \eqref{1n4}.

To get the last part of Theorem \ref{nt1}, as we noticed in \cite[\S
4.5]{DLM04a}, the constants in \cite[(19)]{liu-ma-arosnricdg} will be uniformly bounded
under our condition, thus we can take $C$ in \eqref{1n4}, 
\eqref{n25} and  \eqref{n29}
independent of $h^L$, $d\nu$.
\end{proof}

We have also $C^m$ estimates.

\begin{theorem} \label{nt2} For $m\in \N$, 
    there exists a constant $C>0$ such that for any
    $f\in C^\infty (X)$, $p\in \N$,
\begin{align}\label{n41}
\begin{split}
& \left| Q_{K_p}f- \frac{{\rm Vol}(X, \nu)}{{\rm Vol}(X, dv_X)}\eta f
\right|_{C^m(X)}
\leqslant \frac{C}{p}    \left|f\right|_{C^m(X)}.
\end{split}
\end{align}
Again the constant $C$ here is uniform bounded in the sense after \eqref{1n4}.
\end{theorem} 
\begin{proof}
Now we replace \eqref{n15} by the following equation which is again
from \cite[(19)]{liu-ma-arosnricdg}.
\begin{multline}\label{n42}
\left |\frac{1}{p^{2n}} K_{\omega,p,x_0}(Z,Z')\right.\\
\left.- \Big(1+  \sum_{r=2}^k p^{-r/2}  J^\prime_{r} (\sqrt{p} Z,\sqrt{p} Z')\Big)
e ^{-\pi p |Z-Z^\prime|^2} \Big)\right|_{C^m(X)}\\
\leqslant C p^{-(k+1)/2}
(1+|\sqrt{p} Z|+|\sqrt{p} Z'|)^N
\exp (- C_0 \sqrt{p} |Z-Z'|)+ O(p^{-\infty}).
\end{multline}
Here $C^{m}(X)$ is the $C^{m}$ norm for the parameter $x_0\in X$.
Thus 
\begin{multline}\label{n43}
\left|p^{-n}  K_{\omega,p} f\right.\\
\left. - p^n \int_{|Z^\prime|\leqslant \varepsilon}
\Big(1+  \sum_{r=2}^k p^{-r/2}  J^\prime_{r} (0,\sqrt{p} Z')\Big)
e^{-\pi p |Z^\prime|^2}f_{x_0}(Z^\prime)
dv_X(Z^\prime)\right|_{C^m(X)}\\
\leqslant C p^{-(k+1)/2} \left|f\right|_{C^m(X)}.
\end{multline}

But as in the proof of \cite[Theorem 2.29. (2)]{BeGeVe}, we get
\begin{align}\label{n44}
\begin{split}
&\left| p^n \int_{|Z^\prime|\leqslant \varepsilon} J^\prime_{r} (0,\sqrt{p} Z')
e^{-\pi p |Z^\prime|^2}f_{x_0}(Z^\prime)
dv_X(Z^\prime)\right|_{C^m(X)}
\leqslant C \left|f\right|_{C^m(X)},\\
&\left| p^n \int_{|Z^\prime|\leqslant \varepsilon} 
e^{-\pi p |Z^\prime|^2}f_{x_0}(Z^\prime)
dv_X(Z^\prime)- f(x_{0})\right|_{C^m(X)}
\leqslant C p^{-1} \left|f\right|_{C^m(X)}.
\end{split}\end{align}
From \eqref{n43},  \eqref{n44},  we get
\begin{align}\label{n45}
    \left|p^{-n}  K_{\omega,p} f -f \right|_{C^m(X)}
\leqslant C \left|f\right|_{C^m(X)}. 
\end{align}
Now by  \eqref{n9},  
\begin{align}\label{n46}
    (Q_{K_p}f)(x) =\frac{1}{R_{p}} \eta(x) (K_{\omega,p}f)(x).
\end{align}
From \eqref{n45}, \eqref{n46}, we get \eqref{n41}.

As the constant $C$ in \cite[(19)]{liu-ma-arosnricdg} is uniformly bounded under 
our condition, thus the constant $C$ in \eqref{n42} (and so 
\eqref{n41}) is uniformly bounded.
\end{proof}

\bibliographystyle{plain}
\bibliography{cflow_bib}

{\small \noindent {\tt joel.fine@ulb.ac.be }} \newline
D\'epartment de Math\'ematique,
Universit\'e Libre de Bruxelles CP218,\\ 
Boulevard du Triomphe,
Bruxelles 1050,
Belgique.\\

{\small \noindent {\tt liu@math.ucla.edu}} \newline
Center of Mathematical Science, Zhejiang University, China and\\
Department of Mathematics, UCLA, CA 90095-1555, USA.\\

{\small \noindent {\tt ma@math.jussieu.fr }} \newline
UFR de MathŽmatiques, Case 7012, Site Chevaleret\\
75205 Paris Cedex 13, France

\end{document}